\documentclass[11pt,oneside]{amsart}
\usepackage{amsmath,ifthen, amsfonts, amssymb, srcltx,
   amsopn, textcomp}

\usepackage{hyperref}
\usepackage{graphicx}

\usepackage[small]{caption}

\usepackage{geometry}
\usepackage{graphics}
\usepackage{mathrsfs}
\usepackage{stmaryrd}
\usepackage{tikz}
\usepackage{tikz-cd}
\usepackage{framed}
\usepackage{mathtools}
\usepackage{xfrac}
\usepackage{faktor}
\usepackage{stackengine}
\usepackage{braket}
\usepackage{comment}

\newcommand{\showcomments}{yes}
\renewcommand{\showcomments}{no}

\newcommand{\hidetodo}[1]
{\ifthenelse{\equal{\showcomments}{yes}}%
{#1}
}

\newsavebox{\commentbox}
\newenvironment{com}%
{\ifthenelse{\equal{\showcomments}{yes}}%
{\footnotemark
        \begin{lrbox}{\commentbox}
        \begin{minipage}[t]{1.25in}\raggedright\sffamily\tiny
        \footnotemark[\arabic{footnote}]}
{\begin{lrbox}{\commentbox}}}%
{\ifthenelse{\equal{\showcomments}{yes}}%
{\end{minipage}\end{lrbox}\marginpar{\usebox{\commentbox}}}
{\end{lrbox}}}

\newtheorem{thm}{Theorem}[section]
\newtheorem{lem}[thm]{Lemma}

\newtheorem{cor}[thm]{Corollary}

\newtheorem{prop}[thm]{Proposition}

\theoremstyle{definition}
\newtheorem{defn}[thm]{Definition}
\newtheorem{rem}[thm]{Remark}

\DeclareMathOperator{\Aut}{Aut}

\DeclareMathOperator{\link}{link}

\DeclareMathOperator{\stab}{Stab}

\DeclareMathOperator{\Stab}{Stab}

\newcommand{\bigslant}[2]{{\raisebox{.2em}{$#1$}\left/\raisebox{-.2em}{$#2$}\right.}}

\begin{document}

\title[The Hrushovski Property For Compact Special Cube Complexes]{The Hrushovski Property For Compact Special Cube Complexes}
\author[B.~Abdenbi]{Brahim Abdenbi}
\email{brahim.abdenbi@mail.mcgill.ca  \, \, \, daniel.wise@mcgill.ca}
\author[D.~T.~Wise]{Daniel T. Wise}
          \address{Dept. of Math. \& Stats.\\
                    McGill Univ. \\
                    Montreal, Quebec, Canada H3A 0B9}
\thanks{Research supported by NSERC}
\date{\today}

\maketitle

\begin{com}
{\bf \normalsize COMMENTS\\}
ARE\\
SHOWING!\\
\end{com}

\begin{abstract}
We show that any compact nonpositively curved cube complex $Y$ embeds in a compact nonpositively curved cube complex $R$ where each combinatorial injective partial local isometry of $Y$ extends to an automorphism of $R$. When $Y$ is special and the collection of injective partial local isometries satisfies certain conditions, we show that $R$ can be chosen to be special and the embedding $Y\hookrightarrow R$ can be chosen to be a local isometry.
 \end{abstract}

\section{Introduction}
A well-known theorem of Hrushovski~\cite{MR1194731} states that for any finite graph $X$, there exists a finite graph $Z$ containing $X$ as an induced subgraph with the property that every isomorphism between induced subgraphs of $X$ extends to an automorphism of $Z$.  \textcolor{black}{Hrushovski's motivation stemmed from the study of automorphisms of the countable random graph. This type of property is known as the \textit{Extension Property for Partial Automorphisms} or the \textit{Hrushovski Property}. More generally, a class of spaces $\mathcal{C}$ has the Hrushovski property if for each $X\in \mathcal{C}$, there is 
$Z\in \mathcal{C}$ containing $X$  such that partial isomorphisms between subspaces of $X$ extend to automorphisms of $Z$. Often, the embedding $X\hookrightarrow Z$ is required to have certain properties. For example, in the original Hrushovski property, we require that $X$ be an induced subgraph of $Z$. Our generalization to \textit{special cube complexes} requires the embedding to be a \textit{local isometry}. See Section~\ref{sec:cubes} for definitions. Various classes of spaces were shown to have this property.} \textcolor{black}{For example, the Hrushovski property was established for finite metric spaces by Solecki  \cite{MR2255813} and independently by Vershick \cite{MR2435153}. More recently, it was established for generalized metric spaces by Conant \cite{MR3874663}. Structures of finite relational languages were shown to have this property by Herwig and Lascar  \cite{MR1621745}, and by Hodkinson and Otto \cite{MR2005955}. Similarly, the Hrushovski property was established for various classes of graphs by Herwig \cite{MR1658539} and for  hypertournaments by Huang, Pawliuk, Sabok, and Wise \cite{MR4048719}, provided that the partial isomorphisms have disjoint domains and ranges. We direct the reader to recent surveys by Nguyen Van Th\'e \cite{MR3362231} and Lascar \cite{MR1957017} for detailed discussions of related results.} 

\textcolor{black}{In this paper, we establish the Hrushovski property for compact nonpositively curved cube complexes, and under certain conditions, for compact special cube} \textcolor{black}{complexes as well. Since their introduction by Gromov \cite{Gromov87}, nonpositively curved cube complexes have been used in various lines of research within geometric group theory and played an important role in recent developments. Their connection to group theory is due to a construction by Sageev \cite{Sageev95} that takes as input a group $G$ and a codimension-$1$ subgroup $H\subset G$, and outputs a CAT$(0)$ cube complex $X$ (that is, $X$ is nonpositively curved and simply connected) with a nontrivial group action $G\curvearrowright X$.}

\textcolor{black}{Special cube complexes were introduced by Haglund-Wise \cite{HaglundWiseSpecial} as nonpositively curved cube complexes whose hyperplanes avoid certain configurations. \textit{Hyperplanes} are connected subspaces built from midcubes. See Section~\ref{sec:cubes} for details. It was shown in \cite{HaglundWiseSpecial} that a nonpositively curved cube complex $X$ is \textit{special} if and only if $\pi_1X$ embeds in a \textit{raag} (right angled Artin group). Raags are rather ``simple" groups that are known to be linear and have many desirable residual properties. See Charney \cite{Charney2007} for a brief introduction to raags. This remarkable connection led to several interesting results in combinatorial group theory and low-dimensional topology (see for example \cite{AgolGrovesManning2012} and \cite{WiseCBMS2012}), and makes special cube complexes a particularly natural generalization of graphs. So, the question of whether special cube complexes have the Hrushovski property is a natural one to pursue.}

\textcolor{black}{Our approach to this question has similarities with the work of Herwig and Lascar who showed in ~\cite{MR1621745} that the Hrushovski property for certain spaces is related to the \textit{profinite topology} of free groups. This is the topology generated by finite index subgroups. We make extensive use of this relationship, albeit using a different construction, namely the \textit{horizontal quotients} of \textit{graphs of spaces}. A topological space $X$ is a \textit{graph of spaces} if there is a graph $\Gamma_X$ and quotient map  $X\rightarrow \Gamma_X$ that induces a decomposition of $X$ into \textit{vertex-spaces} $\left\{X_v\ :\ v\in \Gamma_X^0\right\}$ and \textit{thick edge-spaces} $\left\{X_e\times I\ :\ e\in \Gamma_X^1\right\}$ where $I=[-1,1]$ and each $X_v$  (resp. $X_e\times I$) is the preimage of a vertex $v$ (resp. edge $e$) of $\Gamma_X$. The \textit{horizontal quotient} $X^{E}$ is obtained from $X$ by collapsing all thick edge-spaces along their $I$ factor. The general idea is outlined below. See Figure~\ref{hr3.png}.
\begin{figure}[t]\centering
\includegraphics[width=1\textwidth]{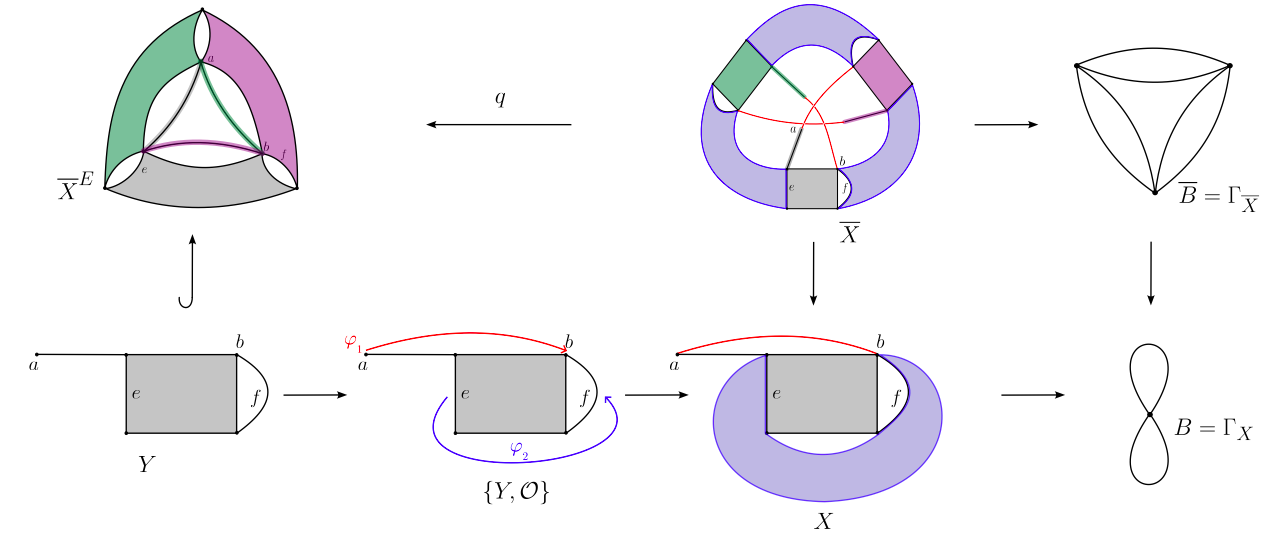}
\caption[example]{\label{hr3.png}
A space $Y$ with partial isomorphisms $\varphi_1, \varphi_2$ (bottom left) embeds in $\overline{X}^{E}$ (top left), where $\varphi_1, \varphi_2$ extend to rotations of $\overline{X}^{E}$.}
\end{figure}
Starting with a space $Y$ and a collection of partial isomorphisms $\mathcal{O}=\left\{\varphi_i:Y_i\subset Y\rightarrow Y\right\}_i$, we build the \textit{realization} $X$ of the pair $\left\{Y,\mathcal{O}\right\}$ by attaching the mapping cylinders of each $\varphi_i$ to $Y$ in the obvious way. The resulting space $X$ has a graph of spaces structure $X\rightarrow \Gamma_X=B$ where $B$ is a bouquet of circles. Each covering map $\overline{B}\rightarrow B$ induces a covering map $\overline{X}\rightarrow X$ where $\overline{X}\rightarrow \overline{B}=\Gamma_{\overline{X}}$ is a graph of spaces decomposition so that the following diagram commutes:}
\begin{center}
    \begin{tikzcd}
\overline{X} \arrow[d] \arrow[r] & \overline{B} \arrow[d] \\
X \arrow[r]                      & B                     
\end{tikzcd}
\end{center}
\textcolor{black}{For each automorphism of $\overline{B}$, there is an automorphism of $\overline{X}$ which preserves the commutativity of the above diagram. We then take the horizontal quotient $\overline{X}\rightarrow \overline{X}^{E}$ and observe that any automorphism of $\overline{X}$ descends to an automorphism of $\overline{X}^{E}$. The horizontal quotient amounts to gluing the vertex-spaces by identifying them along copies of domains and ranges of the partial isomorphisms. Furthermore, 
with the right choice of $\overline{B}$ (and thus of $\overline{X}$), we can ensure that $Y$ embeds in $\overline{X}^{E}$ and each partial isomorphism of $Y\subset \overline{X}^{E}$ extends to an automorphism of $\overline{X}^{E}$. This is achieved using the following theorem  by Ribes-Zalesskii \cite{RibesZalesskii93}:}

\begin{thm}[Ribes-Zalesskii \cite{RibesZalesskii93}]\label{thm:costthm}
Let $H_1,\ldots,H_m$ be finitely generated subgroups of a free group $F$. Then $H_1H_2\cdots H_m$ is closed in the profinite topology.
\end{thm}
\textcolor{black}{This generalizes a result of Hall \cite{Hall49} on finitely generated subgroups of free groups.
Theorem~\ref{thm:costthm} provides a way of choosing an appropriate regular cover $\overline{X}$ so that $\overline{X}^{E}$ has the desired properties. This, however, comes at the cost of imposing some finiteness condition on $Y$, hence the ``compactness" requirement in our results.}

\textcolor{black}{In the case of special cube complexes, we additionally wish to have $\overline{X}^{E}$ be special and the embedding $Y\rightarrow \overline{X}^{E}$ be a local isometry. To this end, we require that the collection of the partial local isometries be ``controlled''. For example, we require that each $\varphi_i$ map non-crossing hyperplanes to non-crossing hyperplanes. See Definition~\ref{defn:controlled1} and Definition~\ref{defn:controlled} for details. This is a necessary condition in order to avoid creating artificial pathologies that would make specialness fail in $\overline{X}^{E}$ regardless of the choice of the covering space $\overline{X}$.}

\textcolor{black}{In our construction, if $\overline{\Phi}^{E}_i$ is an automorphism of $\overline{X}^{E}$ that extends a partial isomorphism $\varphi_i$, and if $\varphi_i$ restricts to $\varphi_j$, then $\overline{\Phi}^{E}_i$ also extends $\varphi_j$. Thus, our construction can be made more efficient by removing any $\varphi_j\in \mathcal{O}$ that is a restriction of another $\varphi_i\in \mathcal{O}$. One drawback of this construction is the fact that the automorphisms of $\overline{X}^{E}$ extending $\varphi$ and $\varphi^{-1}$ are not necessarily inverses of each other. Similarly, a composition of partial isomorphisms does not necessarily extend to the composition of the corresponding automorphisms. One can possibly remedy this by taking further quotients of $\overline{X}^{E}$ to ensure the desired automorphisms are equal. We have not explored this avenue.}

\textcolor{black}{Several statements in this work could be generalized in various directions. For example, this construction works on nonpositively curved metric spaces. See Remark~\ref{metric}. However, we focus on compact nonpositively curved cube complexes and partial local isometries. This is arguably a natural generalization of the original statement about graphs. 
 The main results in this text are:}
 \begin{thm}\label{nthm:001}
Let $Y$ be a compact nonpositively curved cube complex and let $\mathcal{O}$ be the set of injective partial local isometries of $Y$. Then $Y$ embeds in a compact nonpositively curved cube complex $R$ where each $\varphi\in \mathcal{O}$ extends to an automorphism $\Phi\in \Aut\left(R\right)$.
\end{thm}

\begin{thm}\label{thm:0002}
Let $Y$ be a compact special cube complex and let $\mathcal{O}$ be a controlled collection of injective partial local isometries of $Y$. Then  there exists a compact special cube complex $R$ containing $Y$ as a locally convex subcomplex such that each $\varphi\in \mathcal{O}$ extends to an automorphism $\Phi\in \Aut\left(R\right)$.
\end{thm}
 In Sections~\ref{sec:cubes} and \ref{graphs} we provide definitions and background. Section~\ref{separability} uses subgroup separability of free groups to find finite covers whose horizontal quotients have certain desired properties. In Section~\ref{sec:construction} we prove Theorem~\ref{thm:001} and Theorem~\ref{thm:002}.

\vspace{5mm}

 \textbf{Acknowledgement}: We are extremely grateful to Fr\'ed\'eric Haglund and Piotr Przytycki for their helpful feedback and corrections. We also thank the anonymous referees for many corrections and suggestions that improved the text.
 \section{Special Cube Complexes}\label{sec:cubes}
\subsection{Cube Complexes}
An $n$-\textit{cube} is a copy of $I^n$ where $I=\left[-1,1\right]\subset \mathbb{R}$ and $n\geq0$. Its faces are restrictions of some coordinates to $-1$ or $1$. A \textit{cube complex} is a cell complex built from cubes glued together along their faces. The dimension of a cube complex is the supremum of the dimensions of the cubes contained in it. \\
Let $v=\left(\epsilon_i\right)_{i=1}^n$ be a vertex of $I^n$; so each $\epsilon_i=\pm 1$. The $v$-\textit{corner} of $I^n$ is the simplex spanned by $\left\{w_j\right\}_{j=1}^{n}$ where each $w_j$ is obtained from $v$ by replacing $\epsilon_j$ by $\dfrac{\epsilon_j}{2}$.\\
Let $X$ be a cube complex and $C\subset X$ be the image of a map $I^n\rightarrow X$. An $x$-\textit{corner} of $C$ for $x\in X^0$ is the union of images of $v$-corners of $I^n$ where $v\mapsto x$.

In general, if $J=\displaystyle\prod_{i=1}^n\epsilon_i$ is an $m$-dimensional subcube of $I^n$ where $$\epsilon_i\in \Bigg\{\left\{-1\right\},\left\{1\right\},\left[-1,1\right]\Bigg\}$$ then the $J$-\textit{corner} of $I^n$ is the simplex spanned by the points $\left\{w_j\right\}_{j=1}^{n-m}$ obtained from $J$ as follows:

Given the \textit{center of mass} of $J$, denoted by $v=\left(t_k\right)_{k=1}^n$ where

\[   t_k=\left\{
\begin{array}{ll}
      0\quad  \text{if}\quad \epsilon_k=\left[-1,1\right] \\
      1\quad \text{if}\quad \epsilon_k=\left\{1\right\}\\
      -1\quad \text{if}\quad \epsilon_k=\left\{-1\right\}\\
\end{array},
\right. \]
the point $w_j$ is obtained from $v$ by replacing the $j^{\text{th}}$ nonzero coordinate $t$ with $\dfrac{t}{2}$. Note that each point $w_j\in \left\{w_j\right\}_{j=1}^{n-m}$ corresponds to a cube containing $J$ as a codimension-$1$ subcube. 

Let $D$ be a subcube of an $n$-cube $C$ of $X$. A $D$-\textit{corner} of $C$ is the image of a $J$-corner of $I^n$ under a map $I^n\rightarrow X$, where $\left(I^n,J\right)\rightarrow\left(C,D\right)$.
The \textit{link} of $D$ in $X$, denoted by $\link_X\left(D\right)$ is the union of all $D$-corners of cubes containing $D$. Note that $\link_X\left(D\right)$ is a simplex complex and it is a subspace of $X$ but not a subcomplex. We  write $\link\left(D\right)$ instead of $\link_X\left(D\right)$ when $X$ is clear from context. A cube complex $X$ is \textit{simple} if the link of each cube in $X$ is simplicial. 
\begin{lem}\label{simple}
A cube complex $X$ is simple if the link of each cube of $X$ has no loops and no bigons.
\end{lem}
\begin{proof}
Let $D\subset X$ be an $n$-cube. Let $\sigma_1$ and $\sigma_2$ be distinct $m$-simplices in $\link\left(D\right)$ with $\sigma_1\cap \sigma_2\neq \emptyset$ and $m\geq 1$. If $\sigma_1$ is not embedded then $\link\left(D\right)$ has a loop. If $\partial\sigma_1=\partial\sigma_2$, then there exists an $\left(n+m-1\right)$-cube $Y\supset D$ such that $\link\left(Y\right)$ contains a bigon. Indeed, the case $m=1$
corresponds to $Y=D$ with a bigon in $\link\left(D\right)$. For $m\geq 2$, the $m$-simplices $\sigma_1$ and $\sigma_2$ are $D$-corners of distinct $\left(n+m+1\right)$-cubes $A_1$ and $A_2$ intersecting along their faces. An $\left(m-2\right)$-simplex $\Delta\subset \sigma_1\cap\sigma_2$ is a $D$-corner of an $\left(n+m-1\right)$-cube $Y\supset D$. Moreover, two distinct $\left(m-1\right)$-simplices containing $\Delta$ are $D$-corners of distinct $\left(n+m\right)$-cubes $B\supset Y$ and $B'\supset Y$ that are shared faces of $A_1$ and $A_2$. We can see that the $Y$-corners of $B$ and $B'$ are $0$-simplices that are boundaries of the $1$-simplices corresponding to the $Y$-corners of $A_1$ and $A_2$.
\end{proof}
\subsection{Nonpositive curvature}
A simple cube complex $X$ is nonpositively curved if it satisfies Gromov's no-$\triangle$ property \cite{Gromov87}, which requires that $3$-cycles in $\link\left(D\right)$ bound $2$-simplices for each cube $D\subset X$. An equivalent criterion for nonpositive curvature states  that a cube complex is nonpositively curved if the links of its $0$-cubes are flag. A simplicial complex is \textit{flag} if any collection of $\left(n+1\right)$ pairwise adjacent $0$-simplices spans an $n$-simplex. 
\subsection{Local Isometries}\label{Local Isometries} A subcomplex $K$ of a simplicial complex $L$ is \textit{full} if any simplex of $L$ whose $0$-simplices lie in $K$ is itself in $K$. A  subcubecomplex $A\subset B$ is \textit{locally convex} if $\link_A\left(x\right)\subset \link_B\left(x\right)$ is a full subcomplex for every $0$-cube $x\in A$.

A map $X\rightarrow Y$ of cube complexes is \textit{combinatorial} if open cells are mapped homeomorphically to open cells, \textcolor{black}{where each homeomorphism is an isometry}. It is \textit{cubical} if \textcolor{black}{for each $0\leq k\leq \dim\left(X\right)$}, the $k$-skeleton of $X$ is mapped to the $k$-skeleton of $Y$. A combinatorial map $ \Phi: X\rightarrow Y$ is an \textit{immersion} if the restriction $\link\left(x\right)\rightarrow \link\left(\Phi\left(x\right)\right)$ is an embedding for each $0$-cube $x\in X$. \textcolor{black}{If $X$ and $Y$ are nonpositively curved and} $\link\left(x\right)$ embeds as a full subcomplex of $\link\left(\Phi\left(x\right)\right)$ then $\Phi$ is a \textit{local isometry}. \textcolor{black}{Equivalently}, a combinatorial locally injective map $\Phi: X\rightarrow Y$ of nonpositively curved cube complexes is a local isometry if \textcolor{black}{$\Phi$} has \textit{no missing squares} in the sense that if two $1$-cubes $a_1, a_2$ at a $0$-cube $x$ map to $\Phi\left(a_1\right), \Phi\left(a_2\right)$ that bound the corner of a $2$-cube at $\Phi\left(x\right)$, then $a_1, a_2$ already bound the corner of a $2$-cube at $x$. \textcolor{black}{Note that when $\Phi$ is an injective local isometry, $\Phi\left(X\right)$ embeds as a locally convex subcomplex of $Y$}.
\subsection{Immersed Hyperplanes}A \textit{midcube} of an $n$-cube is the subspace obtained by restricting one coordinate to $0$. Note that a midcube of an $n$-cube is isometric to an $\left(n-1\right)$-cube. An \textit{immersed hyperplane} \textcolor{black}{$H$} in a nonpositively curved cube complex $X$ is a component of the cube complex $\bigslant{M}{\sim}$  where $M$ denotes the disjoint union of midcubes of $X$ and $\sim$ is the equivalence relation induced by identifying faces of midcubes under the inclusion map into $X$.  A $1$-cube of $X$ is \textit{dual} to $H$ if its midcube is in $H$. 
\textcolor{black}{We note that the edges dual to $H$ form an equivalence class generated by \textit{elementary parallelisms} of $1$-cubes, where two $1$-cubes are \textit{elementary parallel} if they appear on opposite sides of a $2$-cube}. \textcolor{black}{The \textit{carrier} of $H$, denoted by $N\left(H\right)$, is the cubical neighborhood of $H$ formed by the union of the closed cubes whose intersection with $H$ is nonempty.}

\subsection{Special Cube Complexes}\label{sec:special} An immersed hyperplane $H$ in $X$ \textit{self-crosses} if it contains two distinct midcubes from the same cube. It is \textit{two-sided} if the combinatorial immersion $H\rightarrow X$ extends to $H\times I\rightarrow X$. In this case, the $1$-cubes dual to $H$ can be oriented in such a way that any two dual $1$-cubes lying in the same $2$-cube are oriented in the same direction. An immersed hyperplane that is not two-sided is \textit{one-sided}. $H$ \textit{self-osculates} if it is dual to two oriented $1$-cubes that share the same initial or terminal $0$-cube \textcolor{black}{and do not form a corner of a $2$-cube}. Two \textcolor{black}{distinct} immersed hyperplanes, $H, H'$, \textit{cross} if they contain distinct midcubes of the same cube. They \textit{osculate} if they are dual to two $1$-cubes that share a $0$-cube \textcolor{black}{and do not form a corner of a $2$-cube}. Two \textcolor{black}{distinct} immersed hyperplanes \textit{inter-osculate} if they both cross and osculate. See Figure~\ref{fig:patho}.
\begin{figure}[t]\centering
\includegraphics[width=.7\textwidth]{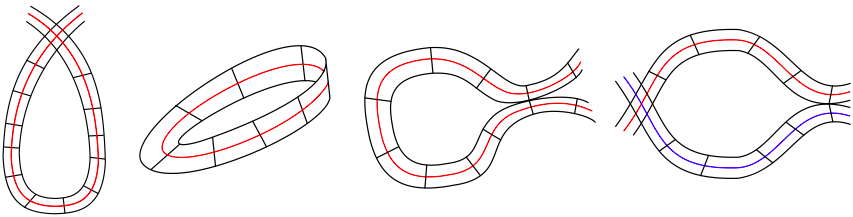}
\caption[Self-crossing, one-sidedness, self-osculation, and inter-osculation.]{\label{fig:patho}
From left to right: Self-crossing, one-sidedness, self-osculation, and inter-osculation.}
\end{figure}
A nonpositively curved cube complex is \textit{special} if it satisfies the following:
\begin{enumerate}
\item No immersed hyperplane self-crosses;
\item No immersed hyperplane is one-sided;
\item No immersed hyperplane self-osculates;
\item No two immersed hyperplanes inter-osculate.
\end{enumerate}

\section{Horizontal Quotient of a Graph of Spaces}\label{graphs}
\subsection{Graph of Spaces} An \textit{undirected graph} $\Gamma\left(V,E\right)$ is a $1$-dimensional $CW$-complex whose \textit{vertices} and \textit{edges}, denoted by  $V=\Gamma^0$ and  $E=\Gamma^1$, are the $0$-cells and open $1$-cells, respectively. There exist two \textit{incidence} maps $\tau_1, \tau_2 : E\rightarrow V$ mapping each edge $e\in E$ to its \textit{boundary vertices}, $\tau_1\left(e\right),\ \tau_2\left(e\right)$ called initial and terminal vertex, respectively. A \textit{graph of spaces} $X$ with underlying graph $\Gamma\left(V,E\right)$, \textit{vertex-spaces} $\left\{X_v\right\}_{v\in V}$, and \textit{\textcolor{black}{thick} edge-spaces} $\left\{X_e\textcolor{black}{\times I}\right\}_{e\in E}$ is a topological space $X$  obtained as a quotient of $\left\{X_v\right\}_{v\in V}$ and $\left\{X_e\textcolor{black}{\times I}\right\}_{e\in E}$ in the following manner: for each edge $e\in E$ with boundary vertices $v_1=\tau_1\left(e\right), v_2=\tau_2\left(e\right)$, the corresponding \textcolor{black}{thick} edge-space $X_e\times I$ is attached to the vertex-spaces $X_{v_1}, X_{v_2}$ via \textit{attaching maps} which are also denoted by $\tau_1: X_e\times \left\{-1\right\}\rightarrow X_{v_1}$ and $\tau_2: X_e\times \left\{1\right\}\rightarrow X_{v_2}$. For simplicity, the isomorphic subspaces $X_e\times \left\{-1\right\}\subset X_{v_1}$ and $X_e\times \left\{1\right\}\subset X_{v_2}$ are referred to as \textit{edge-spaces} of $X_{v_1}$ and $X_{v_2}$, respectively. In this text, we always assume \textcolor{black}{$X_e$ is connected} and the attaching maps of \textcolor{black}{$X_e\times I$} are injective \textcolor{black}{(and cubical when $X$ is a graph of cube complexes)}. The graph  $\Gamma\left(V,E\right)$ is the quotient of $X$ obtained by mapping $X_v$ to $v$ and $X_e\times \left(-1,1\right)$ to $e$ for each $v\in V$ and $e\in E$. We will henceforth denote a graph of spaces $X$ with underlying graph $\Gamma_X$ by the corresponding canonical quotient map $X\rightarrow \Gamma_X$. 
\subsection{Horizontal Quotient}Let $X\rightarrow \Gamma_X$ be a graph of spaces \textcolor{black}{and let $E$ be the edge set of $\Gamma_X$}. Given an edge $e\in E$, let $\sim_e$ be the equivalence relation on $X_e\times I$ where for all $s,t \in I=\left[-1,1\right]$, we have $\left(x,t\right)\sim_e \left(y,s\right)$ if and only if $x=y$. Let $X^{e}=X/\sim_e$ be the corresponding quotient. The \textit{horizontal quotient} of $X$ along the edge $e$, denoted by $q_e: X\rightarrow X^e$, is the quotient map $X\rightarrow X^{e}=X/\sim_e$. In general, if $E'=\left\{e_1,\ldots,e_n\right\}\subset E$, then
the horizontal quotient of $X$ along $E'$ is \textcolor{black}{the quotient $X\rightarrow X^{E'}=X/\sim_{E'}$ where $\sim_{E'}$ is the equivalence relation spanned by \textcolor{black}{the union of all the relations} $\sim_e$ for $e\in E'$}. When $E'=E$, we call $X^{E}$ \textit{the seamless graph of spaces} associated with $X$ and the corresponding map is the \textit{horizontal quotient} which we denote by $q:X\rightarrow X^{E}$. (This terminology was introduced in \cite{HaglundWiseVHification}). \textcolor{black}{Note that the letter $E$ in $X^{E}$ is generic in the sense that it refers to the set of all edges of a given graph. For example, given two graphs of spaces $X\rightarrow \Gamma_X$ and $Y\rightarrow \Gamma_Y$, their horizontal quotients will both be denoted by $X^{E}$ and $Y^{E}$, respectively, even when $\Gamma_X\neq \Gamma_Y$}. The horizontal quotient $q$ is \textit{strict} if the restriction of $q$ to each vertex-space is an embedding. The E-\textit{parallelism class} of a subset $A\subset X$ is $q^{-1}\left(q\left(A\right)\right)$, that is, the set of all points of $X$ mapping to $q\left(A\right)$. When $A$ is a \textcolor{black}{point}, $q^{-1}\left(q\left(A\right)\right)$ is the \textit{horizontal graph} associated to $A$. Note that the restriction of the map $X\rightarrow \Gamma_X$ to a horizontal graph in $X$ is an immersion since the attaching maps of \textcolor{black}{thick} edge-spaces are embeddings. \textcolor{black}{In particular, if $X\rightarrow \Gamma_X$ is a tree of spaces, then $q$ is strict since the horizontal graphs are trees that intersect each vertex-space of $X$ in at most one point.} When $X$ is a graph of cube complexes, an $n$-cube $C\subset X$ is \textit{vertical} if $q\left(C\right)$ is also an $n$-cube.

\begin{rem}\textcolor{black} {In the case of a graph of cube complexes $X$, we make the following remarks: 
\begin{enumerate}
\item The quotient $X^{E}$ is not necessarily a cube complex as cubes of $X$ may be quotiented to simplices in $X^{E}$.  
\item When $q$ is strict, the restriction of $q$ to each cube $\left(C\times I\right)\subset \left(X_e\times I\right)$, where $C\subset X_e$ corresponds to the orthogonal projection (along $I$) onto $C$. Then $q$ is cubical and $X^{E}$ is a cube complex. Moreover, $q$ preserves the orientation of edges in $X_e$.  
\item When $X$ is nonpositively curved, the horizontal quotient $X^{E}$ is not necessarily nonpositively curved. See Figure~\ref{pathologies11}.
\begin{figure}[t]\centering
\includegraphics[width=.7\textwidth]{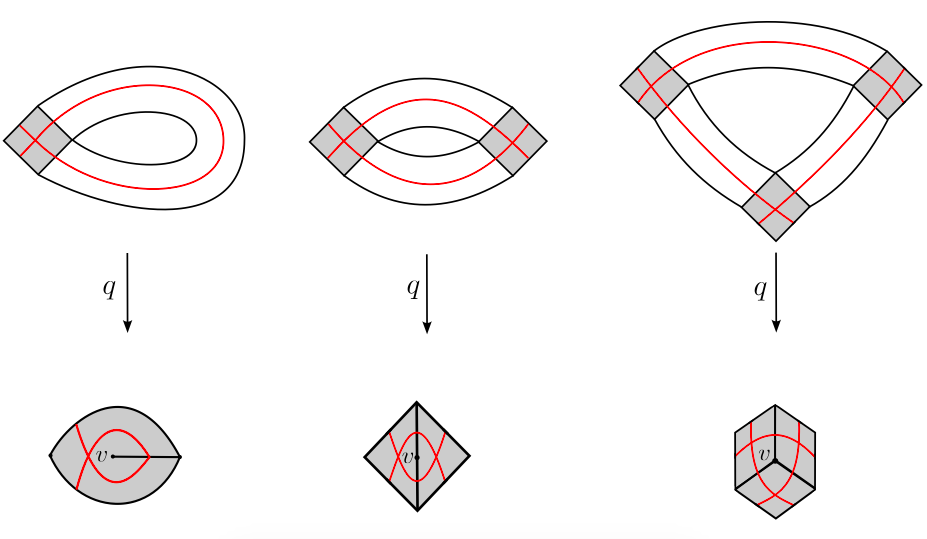}
\caption[]{\label{pathologies11}
The horizontal quotients of these graphs of cube complexes fail the link condition for nonpositive curvature. Indeed, from bottom left to bottom right: $\link\left(v\right)$ is a circle, a bigon, and a $3$-cycle that does not bound a $2$-simplex. In each case, $\link\left(v\right)$ is not flag.}
\end{figure} 
  \end{enumerate} } 
\end{rem}

\begin{lem}\label{lem:003}
Let $X\rightarrow \Gamma_{X}$ be a graph of cube complexes with a strict horizontal quotient. Then
for each immersed hyperplane $U\xrightarrow{f} X^{E}$, there exists an immersed hyperplane $V\xrightarrow{g} X$, with $f\left(U\right)=\left(q\circ g\right)\left(V\right)$. Furthermore, 
\begin{enumerate}
\item if $V$ is two-sided then so is $U$;
\item if $U\xrightarrow{f} X^{E}$ self-crosses, then \textcolor{black}{$V\xrightarrow{g} X$} self-crosses.
\end{enumerate}
Consequently, if the hyperplanes of $X$ are two-sided/embedded then so are the hyperplanes in $X^{E}$.
\end{lem}
\begin{proof}
Since $q$ is strict, it is cubical and so $X^{E}$ is a cube complex. Let $U\xrightarrow{f}X^{E}$ be an immersed hyperplane. Then the parallelism class of $1$-cubes dual to $U$ lifts to a parallelism class of $1$-cubes in $X$. The latter corresponds to an immersed hyperplane $V\xrightarrow{g}X$ that quotients onto $U$, and so $f\left(U\right)=\left(q\circ g\right)\left(V\right)$.  

Now suppose $V\xrightarrow {g}X$ is a two-sided immersed hyperplane. If $g\left(V\right)\subset X_{v}$ for some vertex-space $X_v$, then $q\left(g\left(V\right)\right)$ is two-sided since $q$ is a strict horizontal quotient and thus restricts to an embedding on each vertex-space.  If on the other hand, $g\left(V\right)$ has nonempty intersection with some edge-space $X_e\times I$ attached to vertex-spaces $X_{v_1}, X_{v_2}$, then there exist vertical $1$-cubes $A_1\in X_{v_1}$ and $A_2\in X_{v_2}$ dual to $g\left(V\right)$ that lie on opposite sides of  a $2$-cube $B\subset X_e\times I$. Since $V$ is two-sided, there is a consistent way of orienting $A_1$ and $A_2$ so that their initial points lie on the same $1$-cube of $B$. Taking the horizontal quotient along the edge-space $X_e\times I$,  induces an orientation on $q\left(A_1\right)=q\left(A_2\right)$ consistent with the orientation of the vertical $1$-cubes of $q\left(g\left(V\right)\right)$. By taking consecutive quotients along all the edge-spaces intersecting $g\left(V\right)$, the two-sidedness is preserved at each stage and the claim follows.

Finally, suppose $U\xrightarrow{f} \textcolor{black}{X^{E}}$ is not injective. Then there exists a $2$-cube $S\subset \textcolor{black}{X^{E}}$ where $f\left(U\right)$ self-crosses. The preimage of $S$ contains a $2$-cube where the immersed hyperplane $g\left(V\right)$ self-crosses.
\end{proof}
\begin{rem}\label{rem:001}
\textcolor{black}{Let $X\rightarrow \Gamma_X$ be a graph of cube complexes and let $q:X\rightarrow X^{E}$ be the horizontal quotient. Let $V\xrightarrow{g}X$ be an immersed hyperplane. Then}  $\left(q\circ g\right)\left(V\right)$ is not necessarily the image of an immersed hyperplane in $X^{E}$. \textcolor{black}{Indeed,} not all midcubes of $X$ map to midcubes of $X^{E}$. In particular, \textcolor{black}{each immersed hyperplane $g\left(V\right)=X_e\times \left\{0\right\}\subset \left(X_e\times [-1,1]\right)$ projects to a subcomplex $q\left(g\left(V\right)\right)\subset X^{E}$ that is not a hyperplane}.
\end{rem}

\begin{defn}\label{defn:001} Let $X\rightarrow \Gamma_{X}$ be a graph of cube complexes and $q: X\rightarrow X^{E}$ be the horizontal quotient. \textcolor{black}{Let $x\in X^{E}$ be a $0$-cube and let $q^{-1}\left(x\right)$ be the corresponding horizontal graph. Let $\Gamma_0\subset \Gamma_X$ be the image of $q^{-1}\left(x\right)$ under the quotient $X\rightarrow \Gamma_X$. Let $V_0$ and $E_0$ be the vertices and edges of $\Gamma_0$ and let $\left\{X_v\ :\ v\in V_0\right\}$ and $\left\{X_e\times I\ :\ e\in E_0\right\}$ be the corresponding vertex-spaces and thick edge-spaces in $X$, respectively. Let $\left\{x_1,\ldots\right\}$ be the $0$-cubes of $q^{-1}\left(x\right)$. \textit{The induced graph of links} of $x$ is the graph of spaces $Y\subset X$ with underlying graph $q^{-1}\left(x\right)$, whose vertex-spaces are $\link_{X_{v_i}}\left(x_i\right)$ and whose} \textcolor{black}{thick edge-spaces are $\left(\link_{X_{e_{ij}}}\left(x_i\right)\times I\right)$, where $X_{v_i}\in \left\{X_v\ :\ v\in V_0\right\}$ is the vertex-space containing $x_i$ and $X_{e_{ij}}\times I\in \left\{X_e\times I\ :\ e\in E_0\right\}$ are the thick edge-spaces containing $x_i$}. Note that taking the quotient $X\rightarrow X^{E}$ induces a quotient $Y\rightarrow Y^{E}$ where $\link_{X^{E}}\left(x\right)=Y^{E}$.
\end{defn}
\begin{rem}\label{rem:002}
When the edge-spaces of $X$ are embedded locally convex subcomplexes, the edge-spaces of an induced graph of links are embedded full subcomplexes. However, the vertex-spaces of an induced graph of links are not necessarily connected. 
\end{rem}
\begin{lem}\label{lem:004}
Let $Y=A\displaystyle\cup_{_C}B$ where $A, B$ are simplicial complexes and $C$ embeds as a full subcomplex in $A$ and $B$. Then $Y$ is simplicial and $A$ embeds as a full subcomplex of $Y$.
\end{lem}

\begin{proof}
We show the nonempty intersection of two simplices is a simplex. Let $\sigma_1, \sigma_2\subset Y$ be simplices with $\sigma_1\cap \sigma_2\neq \emptyset$. Each simplex of $Y$ is either in $A$ or in $B$. Suppose $\sigma_1\subset A,\ \sigma_2\subset B$ with $\sigma_1\not\subset B$ and $\sigma_2\not\subset A$. Let $Z$ be the set of  $0$-simplices of $\sigma_1\cap \sigma_2$ and note that $Z\subset C$. Then $Z$ spans simplices $\delta_1\subset A$ and $\delta_2\subset B$. Since $C$ is full in $A$ and $B$, we see that $\delta_1$ and $\delta_2$ are the same simplex of $C$. That is, $\sigma_1\cap \sigma_2$ is a simplex.

To show $A\hookrightarrow Y$ is full, we show that whenever a set of $0$-simplices $S\subset A$ spans a simplex $\Delta$, we have $\Delta \subset A$. Indeed, suppose $\Delta\subset B$, then $S\subset C$. But $C$ is full in $B$ and so $\Delta\subset C\subset A$.
\end{proof}

\begin{lem}\label{lem:005}
Let $Y=A\displaystyle\cup_{_C}B$ where $A, B$ are flag complexes and $C$ embeds as a full subcomplex in $A$ and $B$. Then $Y$ is flag and $A$ embeds as a full subcomplex of $Y$.
\end{lem}
\begin{proof}
$Y$ is simplicial by Lemma~\ref{lem:004}. To show flagness, let $K\subset Y$ be an $n$-clique. We claim that $K\subset A$ or $K\subset B$. We proceed by induction on $n$. The base case $n=0$ is trivial. Suppose the claim holds for all cliques of size $\leq n$ and let $K$ be an $\left(n+1\right)$-clique. By induction, every proper subclique of $K$ lies in either $A$ or $B$. Without loss of generality, let $\sigma_1\in K^0$ be a $0$-simplex with $\sigma_1\notin A$. Then $\sigma_1\in B$ and for any $0$-simplex $\sigma_2\in K^0$, the $1$-simplex $\sigma_1\sigma_2$ lies in $B$. Indeed, if $\sigma_1\sigma_2$ lies in $A$, then $\sigma_1$ lies in $A$ which is a contradiction. Therefore, $\sigma_2\in B$ and so $K^0\subset B$. Moreover, given $0$-simplices $\sigma_2$ and $\sigma_3$ in $K^0$, the $1$-simplex $\sigma_2\sigma_3$ lies in $B$. To see this, suppose $\sigma_2\sigma_3 \in A$. Then and $\sigma_2$ and $\sigma_3$ lie in $A\cap B=C$. But $C$ is full in $A$ and so $\sigma_2\sigma_3\in C\subset B$. Since $B$ is flag, $K$ bounds a simplex.

Let $K\subset Y$ be a clique such that $K^0\subset A$. Then by the previous part, $K^1\subset A$ and it spans a simplex $\Delta\subset A$. Hence $A$ embeds as a full subcomplex of $Y$.
\end{proof}
\begin{lem}\label{lem:006}
Let $Y$ be a tree of spaces where each vertex-space is a flag complex and each edge-space  embeds as a full subcomplex in its vertex-space. Then $Y^{E}$ is flag.
\end{lem}

\begin{proof}
Any failure of flagness arises in a quotient of a finite subtree. Therefore, it suffices to prove the claim for finite trees. This follows by induction from Lemma~\ref{lem:005}. Note that a full subcomplex of a full subcomplex is full.
\end{proof}

\begin{cor}\label{cor:001}
Let $\widehat{X}\rightarrow \Gamma_{\widehat{X}}$ be a tree of nonpositively curved cube complexes where the attaching maps of edge-spaces are injective local isometries. Then $\widehat{X}^{E}$ is nonpositively curved.
\end{cor}
\begin{proof}
Let $x$ be a $0$-cube in $\widehat{X}^{E}$ and let $Y$ be the corresponding induced graph of links \textcolor{black}{with underlying graph $q^{-1}\left(x\right)$}. \textcolor{black}{Since $q^{-1}\left(x\right)$ immerses in $\Gamma_{\widehat{X}}$, it is a tree.} Then $Y$ is a tree of flag complexes with embedded full edge-spaces. By Lemma~\ref{lem:006}, the horizontal quotient $Y^{E}$ is flag, and so is \textcolor{black}{$\link_{\widehat{X}^{E}}\left(x\right)$}.
\end{proof}
\begin{defn}\label{remoteoscu}
Let $X$ be a graph of cube complexes with horizontal quotient $q:X\rightarrow X^{E}$. Let $G$ be a connected \textcolor{black}{subgraph of a} horizontal graph in $X$. Then: 
\begin{enumerate}
\item A hyperplane $U$ \textit{osculates with} $G$ if $U$ is dual to a vertical $1$-cube whose initial or terminal $0$-cube lies in $G$. 
\item A two-sided hyperplane $U$ \textit{self-osculates at $G$} if 
$U$ is dual to oriented vertical $1$-cubes $a$ and $b$ whose initial (or terminal) $0$-cubes $t_a$ and $t_b$ lie in $G$, \textcolor{black}{where $q\left(a\right)$ and $q\left(b\right)$ are not consecutive $1$-cubes of a $2$-cube in $X^{E}$}, and $q\left(a\right)\neq q\left(b\right)$. When $t_a\neq t_b$, the hyperplane $U$ \textit{remotely self-osculates at $G$}, \textcolor{black}{in which case} we say $X$ has \textit{remote self-osculation} \textcolor{black}{(\emph{witnessed by $a$ and $b$})}.
\item A pair of \textcolor{black}{distinct} crossing hyperplanes $U$ and $V$ \textit{inter-osculate at $G$} if there are vertical $1$-cubes $a$ and $b$, with $a$ dual to $U$ and $b$ dual to $V$, with boundary $0$-cubes $t_a$ and $t_b$ lying in $G$, but $q\left(a\right)$ and $q\left(b\right)$ are not consecutive $1$-cubes of a $2$-cube in $X^{E}$. When $t_a\neq t_b$, the hyperplanes $U$ and $V$ \textit{remotely inter-osculate at $G$} \textcolor{black}{in which case} we say $X$ has \textit{remote inter-osculation} \textcolor{black}{(\emph{witnessed by $a$ and $b$})}.
\end{enumerate}
\end{defn}
Note that Definition~\ref{remoteoscu} agrees with the definitions in Section~\ref{sec:special} when $t_a=t_b$. 
\begin{rem}\label{rem:003}
Remote self-osculations and inter-osculations in $X$ \textcolor{black}{are not actual self-osculations and inter-osculations, but they} project to self-osculations/inter-osculations under the horizontal quotient $q:X\rightarrow X^{E}$ whenever $q$ is cubical.
\end{rem}
\begin{lem}\label{lem:007}
Let $X$ be a graph of cube complexes and suppose $X$ \textcolor{black}{has no one-sided, self-crossing, self-osculating, or inter-osculating hyperplanes}. If the horizontal quotient \textcolor{black}{$q: X\rightarrow X^{E}$ is cubical and} $X^{E}$ has self-osculation/inter-osculation then $X$ has remote self-osculation/inter-osculation.
\end{lem}
\begin{proof}
Let $U\xrightarrow{f} X^{E}$ be a self-osculating hyperplane. By Lemma~\ref{lem:003}, there is a hyperplane $V\xrightarrow{g} X$ with $q\circ g\left(V\right)=f\left(U\right)$. \textcolor{black}{Since $X$ has no self-crossing hyperplanes}, \textcolor{black}{$g$ and (hence) $f$ are embeddings}, and so we can identify $U$ and $V$ with their images. Since \textcolor{black}{the hyperplanes of $X$ are $2$-sided, and} $q$ is orientation-preserving, the $1$-cubes of $X^{E}$ can be oriented consistently with the orientations of $1$-cubes of $X$. Let $a_u$ and $b_u$ be distinct oriented $1$-cubes dual to $U$ that share the $0$-cube $t$ where the self-osculation occurs. We can assume without loss of generality that $t$ is the terminal $0$-cube of $a_u$ and $b_u$. Let $a_v$ and $b_v$ be oriented $1$-cubes dual to $V$ and mapping to $a_u$ and $b_u$, respectively.  Let $G=q^{-1}\left(t\right)$ be the horizontal graph mapping to $t$. Let $t_a$ and $t_b$ be terminal points of $a_v$ and $b_v$. See Figure~\ref{t1}. Then $t_a$ and $t_b$ lie in $G$ and \textcolor{black}{since $X$ has no self-osculating hyperplanes}, $t_a\neq t_b$. Since $q\left(a_v\right)=a_u\neq b_u=q\left(b_v\right)$, the hyperplane $V$ remotely self-osculates at $G$.\\
\begin{figure}[t]\centering
\includegraphics[width=.6\textwidth]{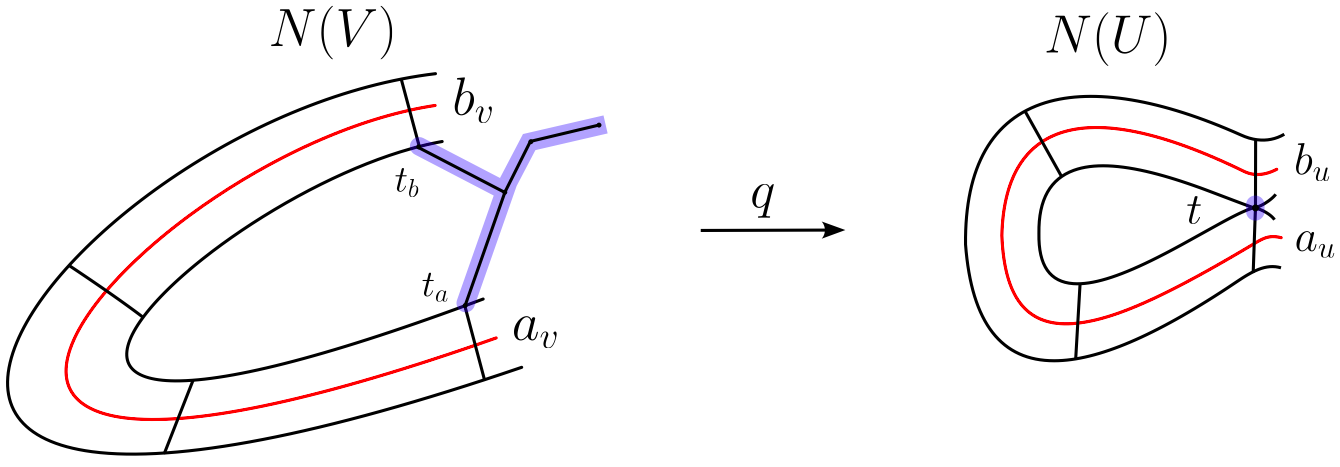}
\caption[ ]{The hyperplane $V$ osculates with $G=q^{-1}\left(t\right)$ at two points $t_a$ and $t_b$.}{\label{t1}
}
\end{figure}
Let $U_1$ and $U_2$ be inter-osculating hyperplanes in $X^{E}$, and let $V_1$ and $V_2$ be the crossing hyperplanes in $X$ mapping to $U_1$ and $U_2$, respectively. Suppose the inter-osculation occurs at $1$-cubes $a_{u_1}$ and $b_{u_2}$ dual to $U_1$ and $U_2$ and meeting at a $0$-cube $t$. Let $a_{v_1}$ and $b_{v_2}$ be $1$-cubes dual to $V_1$ and $V_2$ and mapping to $a_{u_1}$ and $b_{u_2}$, respectively. \textcolor{black}{Since $X$ has no inter-osculating hyperplanes}, $G=q^{-1}\left(t\right)$ is nontrivial and contains the distinct $0$-cubes $t_a$ and $t_b$ of $a_{v_1}$ and $b_{v_2}$. Moreover, since $a_{u_1}$ and $b_{u_2}$ do not form a consecutive pair of edges of a $2$-cube, $V_1$ and $V_2$ remotely inter-osculate at $G$.
\end{proof}

\begin{defn}\label{defn:selfosculation}
Let $X$ be a cube complex. A subcomplex $X'\subset X$ \textit{self-osculates} if there is a hyperplane $U'$ of $X'$ that extends to a hyperplane $U$ of $X$ dual to a $1$-cube \textcolor{black}{whose intersection with $X'$ consists of $0$-cubes.}
\end{defn}
\begin{defn}\label{defn:controlled1}
A graph of cube complexes is \textit{controlled} if for each \textcolor{black}{thick} edge-space $X_e\times I$ attached to vertex-spaces $X_{v_1}$ and $X_{v_2}$, the following hold for each $i\in\left\{1,2\right\}$:
\begin{enumerate}
\item distinct hyperplanes of $X_e$ extend to distinct hyperplanes of $X_{v_i}$ (wall-injectivity);
\item non-crossing hyperplanes of $X_e$ extend to non-crossing hyperplanes of $X_{v_i}$ (cross-injectivity);
\item the edge-space $X_e$ is non self-osculating in $X_{v_i}$.
\end{enumerate}
\end{defn}

\begin{lem}\label{lem:009}
Let $\widehat{X}\rightarrow \Gamma_{\widehat{X}}$ be a controlled tree of cube complexes and suppose each vertex-space of $\widehat{X}$ has embedded hyperplanes. Then each hyperplane $U$ of $\widehat{X}$ \textcolor{black}{dual to a vertical $1$-cube} splits as a tree of spaces $U\rightarrow \Gamma_U$ so that the following diagram commutes:
\begin{center}
\begin{tikzcd}
U \arrow[r] \arrow[d]    & \widehat{X} \arrow[d] \\
\Gamma_U \arrow[r] & \Gamma_{\widehat{X}} 
\end{tikzcd}
\end{center}
Moreover, $\Gamma_U\rightarrow \Gamma_{\widehat{X}}$ is an embedding, 
\textcolor{black}{each hyperplane splits as a tree of connected spaces, each of which embeds in $\widehat{X}$},
and consequently, $U$ embeds in $\widehat{X}$ and $U\cap X_v$ is connected for each vertex-space $X_v\subset \widehat{X}$.
\end{lem}
\begin{proof}
Let $U\rightarrow \Gamma_U$ be a graph of spaces decomposition induced by $\widehat{X}\rightarrow \Gamma_{\widehat{X}}$. \textcolor{black}{Since $U$ is dual to a vertical $1$-cube, $U$ has nonempty intersection with at least one vertex-space.} The vertex-spaces of $U$ are the components of intersections with the vertex-spaces of $\widehat{X}$, and likewise for edge-spaces. Wall-injectivity implies that $U\cap X_v$ is a single hyperplane for each vertex-space $X_v$ intersecting with $U$. So $\Gamma_U\rightarrow \Gamma_{\widehat{X}}$ is an immersion and thus an injection. Therefore, $\Gamma_U$ is a tree and $U\rightarrow \widehat{X}$ is an embedding.
\end{proof}
\begin{lem}\label{lem:010}
Let $\widehat{X}\rightarrow \Gamma_{\widehat{X}}$ be a controlled tree of cube complexes and let $X_e$ be an edge-space in a vertex-space $X_v$. Let $U\subset \widehat{X}$ be an embedded hyperplane dual to a vertical $1$-cube $a\in X_v$. If $a\cap X_e$ \textcolor{black}{consists of} $0$-cubes then $U\cap X_e=\emptyset$. See Figure~\ref{r5}.
\begin{figure}[t]\centering
\includegraphics[width=.7\textwidth]{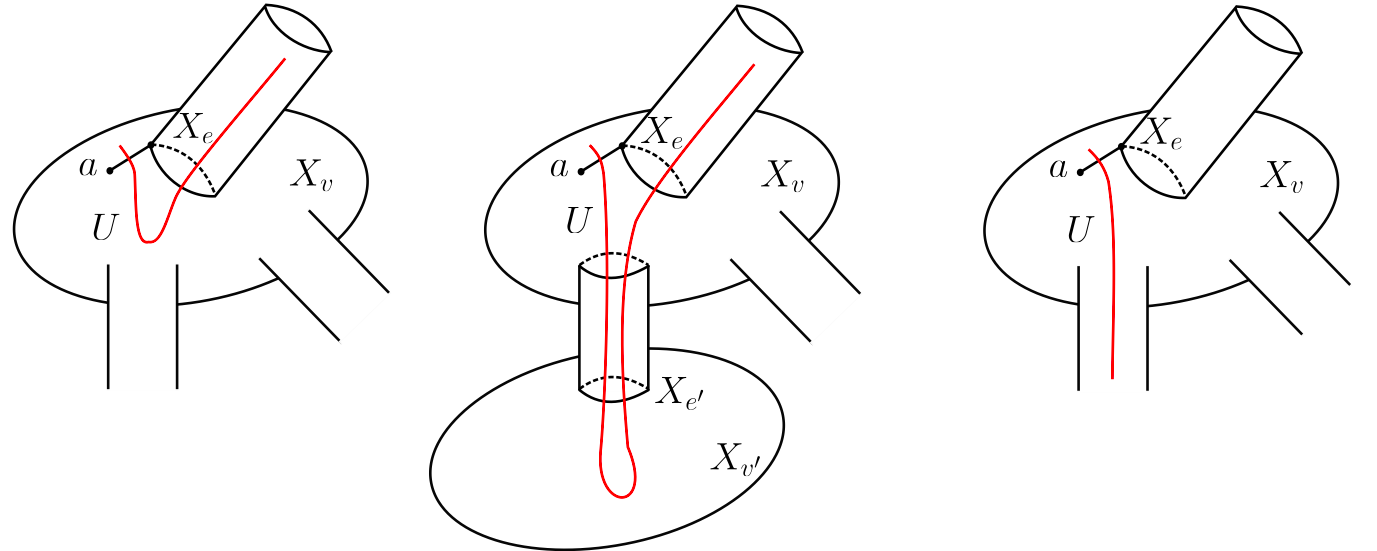}
\caption[ ]{The edge-space $X_e$ osculates with the hyperplane $U$. If $X_e\cap\ U\neq \emptyset$, then either $X_e$ self-osculates (left) or wall-injectivity fails in some edge-space $X_{e'}$ (middle). $X_e\cap U=\emptyset$ (right).}{\label{r5}
}
\end{figure}
\end{lem}
\begin{proof}
By Lemma~\ref{lem:009}, the intersection $U\cap X_v$ is connected. Since $X_e$ is not self-osculating in $X_v$, we have $U\cap X_e=\emptyset$. 
\end{proof}

\begin{lem}\label{disconnected}

Let $X$ be \textcolor{black}{a controlled tree} of cube complexes and let $q:X\rightarrow X^{E}$ be the corresponding horizontal quotient. Let $U$ and $V$ be crossing hyperplanes in $X$ that map into hyperplanes of $X^{E}$. Suppose $X$ is special. Then:
\begin{enumerate}
\item \label{ladder_part1}\textcolor{black}{If $U$ is dual to $1$-cubes $a$ and $b$ with terminal $0$-cubes $t_a$ and $t_b$ respectively, and $\gamma\rightarrow N\left(U\right)$ is a horizontal path from $t_a$ to $t_b$, then $q\left(a\right)=q\left(b\right)$. That is, $a$ and $b$ lie in the same E-parallelism class}.

\item \label{ladder_part2}If $U$ remotely self-osculates at a horizontal graph $G$, then $G\cap N\left(U\right)$ is disconnected. 
\item \label{ladder_part3} If $U$ and $V$ remotely inter-osculate at a horizontal graph $G$, then 

$G\cap \big(N\left(U\right)\cup N\left(V\right)\big)$  is disconnected.
\end{enumerate}
\end{lem}
\textcolor{black}{Note that since $U$ and $V$ map to hyperplanes in $X^{E}$, they do not lie in the middle of a thick edge-space. Thus, any $1$-cube that is dual to $U$ or $V$ is not horizontal. }
\begin{proof}[Proof of Part~\eqref{ladder_part1}] \textcolor{black}{Express $\gamma$ as a concatenation $\gamma=x_0e_1x_1\cdots e_{n-1}x_{n-1}e_nx_n$,  where each $e_i$ is a horizontal $1$-cube joining the $0$-cubes $x_{i-1}$ and $x_i$, with $x_0=t_a$ and $x_n=t_b$. We show that $q\left(a\right) = q\left(b\right)$ by finding a sequence of E-parallel $1$-cubes $a_0=a, a_1, \ldots, a_n=b$ dual to $U$, with terminal $0$-cubes $x_0=t_a, x_1, \ldots, x_n=t_b$. We proceed by induction on $n$.} 

\textcolor{black}{If $n=0$, then $\gamma$ is trivial, and so the $0$-cube $t_a=t_b=x_0$ lies in both $a=a_0$ and $b$ which are dual to $U$. By assumption, $U$ does not self-osculate, and so $a_0=a=b$.}

\textcolor{black}{Suppose the claim holds for $n\geq 0$. Let $\gamma=x_0e_1x_1\cdots e_{n}x_{n}e_{n+1}x_{n+1}$ be a horizontal path from $t_a=x_0$ to $t_b=x_{n+1}$. See Figure~\ref{dual_to_horizontal}. By the inductive assumption, there is a sequence of E-parallel $1$-cubes $a=a_0, a_1,\ldots,a_n$ dual to $U$ and containing the $0$-cubes $x_0,\ldots,x_n$, respectively. Since the attaching maps of edge-spaces are injective, any horizontal $1$-cube $e\subset N\left(U\right)$ lies in a $2$-cube $S\subset N\left(U\right)$ whose opposite $1$-cube is also horizontal. Let $S_{n+1}$ be a $2$-cube containing $e_{n+1}$ with faces $a', e_{n+1}, a_{n+1}, e_{n+1}'$, where $e_{n+1}'$ is the horizontal $1$-cube in $N\left(U\right)$ that is opposite to $e_{n+1}$ in $S_{n+1}$. Note that both $a'$ and $a_{n+1}$ are dual to $U$. Also, note that both $a_n$ and $a'$ contain $x_n$, and both $a_{n+1}$ and $b$ contains $x_{n+1}=t_b$.} 
\begin{figure}[t]\centering
\includegraphics[width=.7\textwidth]{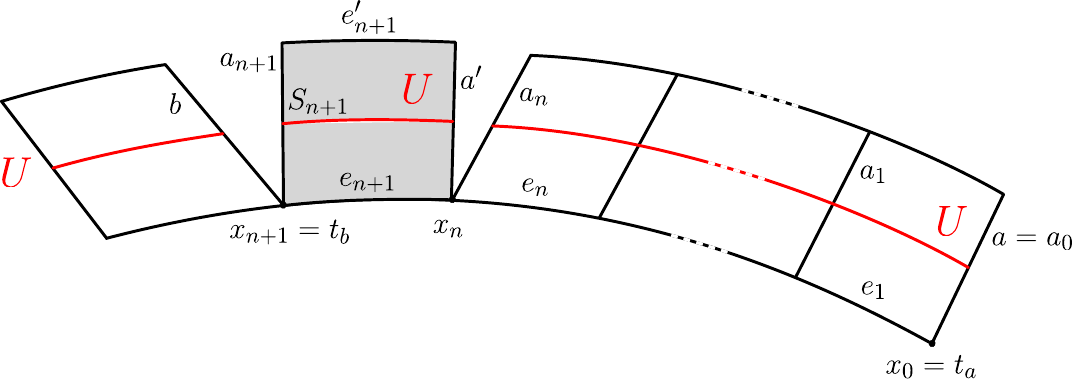}
\caption[ ]{Since $U$ is non self-osculating, $a' = a_n$ and $a_{n+1} = b$.}{\label{dual_to_horizontal}
}
\end{figure}\textcolor{black}{Since $U$ is non-self-osculating, we have $a'=a_n$ and $a_{n+1}=b$. So $q\left(b\right)=q\left(a_n\right)$, and thus $q\left(a\right)=q\left(b\right)$.}
\end{proof}

\begin{proof}[Proof of Part~\eqref{ladder_part2}] \textcolor{black}{Suppose $U$ is a remotely self-osculating hyperplane in $X$. Let $a$ and $b$ be oriented $1$-cubes dual to $U$ with terminal $0$-cubes $t_a$ and $t_b$ in $G\cap N\left(U\right)$, witnessing the remote self-osculation of $U$ as in Definition~\ref{remoteoscu}. Then 
$t_a\neq t_b$ and $q\left(a\right)\neq q\left(b\right)$. By Part~\ref{ladder_part1}, there is no horizontal path in $N\left(U\right)$ joining $t_a$ to $t_b$, and thus $G\cap N\left(U\right)$ is disconnected.}
\end{proof}
\begin{proof}[Proof of Part~\eqref{ladder_part3}]
\textcolor{black}{Suppose $U$ and $V$ are remotely inter-osculating hyperplanes in $X$. Let $a$ and $b$ be the vertical $1$-cubes dual to $U$ and $V$, respectively, with boundary $0$-cubes $t_a\neq t_b$ in $G$, \textcolor{black}{witnessing the remote inter-osculation} as in Definition~\ref{remoteoscu}. We claim that $t_a$ and $t_b$ lie in distinct components of $G\cap \left(N\left(U\right)\cup N\left(V\right)\right)$. Suppose otherwise. Then there is a nontrivial horizontal path $\gamma\rightarrow \left(N\left(U\right)\cup N\left(V\right)\right)$ from $t_a$ to $t_b$. Note that $\gamma$ is the concatenation of horizontal subpaths lying alternatively in $N\left(U\right)$ and $N\left(V\right)$. Assume without loss of generality that  $\gamma=\gamma_1\gamma_2\cdots\gamma_k$ is such a path, where $\gamma_1\subset N\left(U\right)$ and $\gamma_k\subset N\left(V\right)$. Let $y_1, \ldots,y_{k-1}$ be the $0$-cubes in $N\left(U\right)\cap N\left(V\right)$ where the consecutive subpaths of $\gamma$ meet. Then $\left\{y_1, \ldots, y_{k-1}\right\}\neq \emptyset$, since $\gamma$ is a horizontal path from $t_a\in N\left(U\right)$ to $t_b\in N\left(V\right)$. Let $y_0= t_a$ be the initial $0$-cube of $\gamma_1$ and let $y_k=t_b$ be the terminal $0$-cube of $\gamma_k$. Let $a_0=a, a_1,\ldots,a_{k-1}$ be $1$-cubes dual to $U$ with terminal $0$-cubes $y_0,y_1,\ldots,y_{k-1}$, respectively, and let $b_1,\ldots, b_k=b$ be $1$-cubes dual to $V$ with terminal $0$-cubes $y_1,\ldots,y_{k}$, respectively. Since $X$ has no inter-osculating hyperplanes, there are $2$-cubes $S_1,\ldots,S_{k-1}$ where each $S_j$ contains the pair of $1$-cubes $a_j$ and $b_j$ as consecutive faces}. See Figure~\ref{interoscu_horizontal.pdf}. 
\begin{figure}[t]\centering
\includegraphics[width=.8\textwidth]{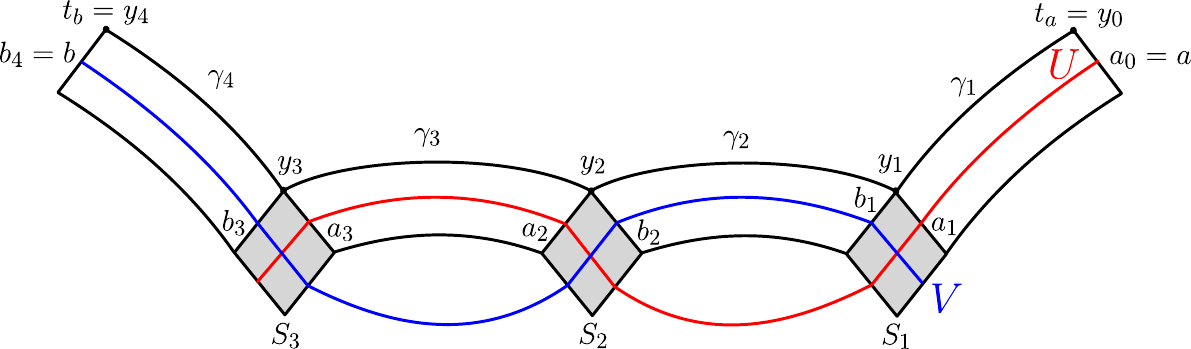}
\caption[ ]{The path $\gamma=\gamma_1\gamma_2\cdots$ in the proof of Lemma~\ref{disconnected}.\eqref{ladder_part3}.}{\label{interoscu_horizontal.pdf}
}
\end{figure}

\textcolor{black}{To reach a contradiction, we show that $q\left(a\right)$ and $q\left(b\right)$ form consecutive 
$1$-cubes of a $2$-cube in $X^{E}$, contradicting Definition~\ref{remoteoscu}. By Part~\ref{ladder_part1}, we have $q\left(a\right)=q\left(a_1\right)$  and $q\left(b\right)=q\left(b_{k-1}\right)$, and so it suffices to show that $q\left(a_1\right)$ and $q\left(b_{k-1}\right)$ are consecutive $1$-cubes of a $2$-cube in $X^{E}$. To this end, we prove the following claim:}

\textcolor{black}{\underline{Claim}: The subpaths of $\gamma$ between $y_1$ and $y_{k-1}$ lie in $N\left(U\right)\cap N\left(V\right)$.}

\begin{proof}[Proof of Claim] \textcolor{black}{Let $\gamma_j$ be a subpath of $\gamma$ from $y_{j-1}$ to $y_j$, with $2 \leq j\leq k-1$. Suppose without loss of generality that $\gamma_j\subset N\left(V\right)$. Express $\gamma_j$ as a concatenation $x_0e_1\cdots e_{n-1}x_{n-1}e_{n}x_{n}$ where each $e_i$ is a horizontal $1$-cube joining the $0$-cubes $x_{i}$ and $x_{i+1}$, with $x_0=y_{j-1}$ and $x_{n}=y_{j}$.  We proceed by induction on $n$.}

\textcolor{black}{If $n=0$, then $\gamma_j = y_j=y_{j-1} \in N\left(U\right)\cap N\left(V\right)$. Suppose the claim holds for $n-1$. Then $x_{n-1}\in N\left(U\right)\cap N\left(V\right)$. Let $f\subset N\left(U\right)$ be a $1$-cube containing $x_{n-1}$ and dual to $U$. Let $X_{e_{n}}\times I$ be the thick edge-space containing $e_n$. Since $x_n=y_j \in N\left(U\right)\cap N\left(V\right)$ and $X$ is a tree of spaces, $U$ must travel through $X_{e_n}\times I$, and so $U\cap X_{e_n}\neq \emptyset$. Since $f\cap X_{e_n} \neq \emptyset$ and $X$ is controlled, by Lemma~\ref{lem:010}, we have $f\subset X_{e_n}$, and thus $f \times I \subset X_{e_n}\times I$. Since the attaching maps of edge-spaces are injective, $e_n\subset f\times I$. So $e_n\subset N\left(U\right)$, and thus $\gamma_j\subset N\left(U\right)$.} \end{proof}
Finally, by the above claim, there is a horizontal path in $N\left(U\right)$ [respectively, $N\left(V\right)$] from $t_{a_1}$ and thus $t_a$ to $t_{a_{j}}$ for any $1 \leq j\leq k-1$ [respectively, from $t_b$ and thus $t_{b_{k-1}}$ to $t_{b_j}$ for any $1 \leq j\leq k-1$]. By Part~\ref{ladder_part1}, we have $q\left(a\right)=q\left(a_j\right)$ [respectively, $q\left(b\right)=q\left(b_j\right)$] for all $1\leq j \leq k-1$. Since there is at least one $2$-cube $S_j$ that lies in a vertex-space, with consecutive $1$-cubes $a_j$ and $b_j$, it follows that $q\left(a\right)$ and $q\left(b\right)$ are consecutive $1$-cubes of the $2$-cube $q\left(S_j\right)$, contradicting Definition~\ref{remoteoscu}.
\end{proof}

\begin{lem}\label{lem:011}
Let $\widehat{X}\xrightarrow{p} \Gamma_{\widehat{X}}$ be a controlled tree of cube complexes \textcolor{black}{that is special}. Then $\widehat{X}$ has no remote self-osculation/inter-osculation.
\end{lem}
\begin{proof}
The horizontal graphs of $\widehat{X}$ are trees \textcolor{black}{that intersect each vertex-space of $\widehat{X}$ in at most one $0$-cube. Suppose $U$ is a hyperplane that remotely self-osculates at a horizontal tree $T$. By Lemma~\ref{disconnected}, $T\cap N\left(U\right)$ is not connected. Let $K_1$ and $K_2$ be components of $T\cap N\left(U\right)$.}
Let $t_1, t_2\in T$ be the closest $0$-cubes in $T$ with $t_1\in K_1$ and $t_2\in K_2$. Let $a_1$ be the $1$-cube dual to $U$ and containing $t_1$. Let $\gamma=e_1\cdots e_n$ be the shortest horizontal path in $T$ from $t_1$ to $t_2$, where each $1$-cube $e_i$ is in the edge-space $X_{e_i}$. Note that $\gamma$ is nontrivial since $t_1\neq t_2$. The $1$-cube $e_1$ with initial $0$-cube $t_1$ does not lie in $N\left(U\right)$ for otherwise, the terminal  $0$-cube of $e_1$ is in $K_1$ and is closer to $t_2$. Then $a_1$ is not in $X_{e_1}$. Since $a_1\cap X_{e_1}$ consists of $0$-cubes, we have by Lemma~\ref{lem:010}, $U\cap X_{e_1}=\emptyset$. On the other hand, since $U$ splits as a graph of spaces $U\rightarrow \Gamma_U$ where $\Gamma_U$ is a subtree of $\Gamma_{\widehat{X}}$, the image $\left(\gamma\rightarrow \Gamma_{\widehat{X}}\right)\hookrightarrow \Gamma_U$ and so $U\cap X_{e_1}\neq \emptyset$ which is a contradiction.

Suppose $U$ and $V$ are hyperplanes that remotely inter-osculate at a horizontal tree $T$. \textcolor{black}{By Lemma~\ref{disconnected}, $T\cap\left(N\left(U\right)\cup N\left(V\right)\right)$ is not connected}. Let $t_1\in N\left(U\right)$ and $t_2\in N\left(V\right)$ be the closest $0$-cubes lying in distinct components of $T\cap \left(N\left(U\right)\cup N\left(V\right)\right)$. Let $\gamma_1=e_1\cdots e_n$ be the nontrivial horizontal path from $t_1$ to $t_2$, where each $1$-cube $e_i$ lies in $X_{e_i}$. Let $a_1$ and $a_2$ be the $1$-cubes dual to $U$ and $V$ and containing $t_1$ and $t_2$, respectively. As in part $(1)$, we have $a_1\notin X_{e_1}$ and $a_2\notin X_{e_n}$, and so by Lemma~\ref{lem:010}, we have $U\cap X_{e_1}=\emptyset$ and $V\cap X_{e_n}=\emptyset$. Since $\widehat{X}$ is a tree of spaces, each pair of vertex-spaces is joined by at most one edge-space. Thus, $U\cap X_{e_1}=\emptyset$ implies $U\cap X_{e_i}=\emptyset$ for all $1\leq i\leq n$. Similarly, $V\cap X_{e_1}=\emptyset$ for all $1\leq i\leq n$. Since $U$ crosses $V$, there is a $0$-cube $x\in N\left(U\right)\cap N\left(V\right)$, and a path $\gamma_2=f_1\cdots f_m$ from $t_2$ to $t_1$ passing through $x$, where $f_j\in \left(N\left(U\right)\cup N\left(V\right)\right)$. The concatenation $\gamma_1\cdot\gamma_2$ projects to a closed path in the tree $\Gamma_{\widehat{X}}$. Since $\gamma_1$ is horizontal, $\gamma_1\xrightarrow{p} \Gamma_{\widehat{X}}$ is an embedding. Hence there is a $1$-cube $f_j\in \gamma_2$ so that $p\left(f_j\right)=p\left(e_1\right)$. If $f_j\in N\left(U\right)$, then $U\cap X_{e_1}\neq \emptyset$ and if $f_j\in N\left(V\right)$, then $V\cap X_{e_1}\neq \emptyset$, both leading to contradictions.
\end{proof}

\begin{prop}\label{proposition}
Let $\widehat{X}\rightarrow \Gamma_{\widehat{X}}$ be a controlled tree of nonpositively curved cube complexes with  embedded locally convex edge-spaces. Let $q:\widehat{X}\rightarrow \widehat{X}^{E}$ be the horizontal quotient. If $\widehat{X}$ is special then so is $\widehat{X}^{E}$.
\end{prop}
\begin{proof} By Corollary ~\ref{cor:001}, $\widehat{X}^{E}$ is nonpositively curved. Since $\widehat{X}$ is a tree of spaces, the horizontal quotient $q:\widehat{X}\rightarrow \widehat{X}^{E}$ is strict. By Lemma \ref{lem:003}, each hyperplane of \textcolor{black}{$\widehat{X}^{E}$} is embedded and two-sided. By Lemma~\ref{lem:007}, self-osculation/inter-osculation in $\widehat{X}^{E}$ arises from remote self-osculation/inter-osculation in $\widehat{X}$. By Lemma~\ref{lem:011}, $\widehat{X}$ has no remote self-osculation/inter-osculation.
\end{proof}

\section{Subgroup Separability}\label{separability}
The collection of finite index cosets of a group $F$ forms a basis for the \textit{profinite topology} on $F$. The multiplication and inversion are continuous with respect to this topology. A subset $S\subset F$ is \textit{separable} if it is closed in the profinite topology. A subgroup $H\subset F$ is separable if and only if $H$ is the intersection of finite index subgroups.

\begin{thm}[Ribes-Zalesskii \cite{RibesZalesskii93}]\label{thm:RZ}
Let $H_1,\ldots,H_m$ be finitely generated subgroups of a free group $F$. Then $H_1H_2\cdots H_m$ is closed in the profinite topology.
\end{thm}
It follows that $g_1H_1g_2H_2\cdots g_mH_m$ is also closed in the profinite topology, for finitely generated subgroups $H_i\subset F$ and $g_i\in F$ with $1\leq i\leq m$.

\textcolor{black}{Starting with a tree of nonpositively curved cube complexes $\widehat{X}\rightarrow \Gamma_{\widehat{X}}$ and using separability properties of the free group action on $\widehat{X}$, we find compact quotients $\widehat{X}\rightarrow \overline{X}$ where the horizontal quotient $\overline{X}\rightarrow \overline{X}^{E}$ is cubical, $\overline{X}^{E}$ is nonpositively curved with well-behaved hyperplanes whenever $\widehat{X}$ is controlled and special.}
\begin{lem}\label{lem:012}
Let $X\rightarrow \Gamma_X$ be a compact graph of cube complexes with one vertex-space $Y$. Then $X$ has a finite regular cover $\overline{X}$ such that:
\begin{enumerate}
\item\label{property1} $\overline{X}\rightarrow \Gamma_{\overline{X}}$ is a graph of spaces whose vertex-spaces are isomorphic to $Y$;
\item\label{property2} The  horizontal quotient $\overline{X}\rightarrow \overline{X}^{E}$ is strict.
\end{enumerate}
\textcolor{black}{Furthermore, any finite regular cover $\overline{X}'\rightarrow \overline{X}$ induced by a cover of underlying graphs $\Gamma_{\overline{X}}'\rightarrow \Gamma_{\overline{X}}$ satisfies properties~\ref{property1} and \ref{property2}.}
\end{lem}

\begin{proof}We find a covering space that splits as a graph of cube complexes with vertex-spaces isomorphic to $Y$ and whose horizontal quotient is strict.

The underlying graph $X\rightarrow \Gamma_X$ is a bouquet of circles. Let $\widetilde{\Gamma_X}\rightarrow \Gamma_X$ be the universal covering map and let $\widehat{X}\rightarrow X$ be the corresponding covering map so that the following diagram commutes: 
\begin{center}
\begin{tikzcd}
\widehat{X} \arrow[r] \arrow[d] & \Gamma_{\widehat{X}}=\widetilde{\Gamma_X} \arrow[d] \\
X \arrow[r]                      & \Gamma_X              
\end{tikzcd}
\end{center}
Then $\pi_1\Gamma_X$ acts freely and cocompactly on $\widehat{X}$. Let  $N\subset \pi_1\Gamma_X$ be a finite index normal subgroup, and let $N\backslash\widehat{X}=\overline{X}\rightarrow X$ be the covering map induced by $N\backslash\Gamma_{\widehat{X}}=\Gamma_{\overline{X}}\rightarrow \Gamma_X$ so that the following diagram commutes:
\begin{center}
\begin{tikzcd}
\overline{X} \arrow[r] \arrow[d] & \Gamma_{\overline{X}} \arrow[d] \\
X \arrow[r]                      & \Gamma_X              
\end{tikzcd}
\end{center}
Then $\overline{X}$ is a graph of cube complexes where each vertex-space is isomorphic to $Y$. 

We need to choose $\overline{X}$, and thus $N$, so that no vertex-space has two points in the same E-parallelism class. In our cubical setting, it is sufficient to ensure that no two $0$-cubes of a vertex-space of $\overline{X}$ lie in the same E-parallelism class. \textcolor{black}{Recall that the attaching maps of edge-spaces are assumed to be injective}.

By compactness, there are finitely many $0$-cubes $\left\{C_i\right\}_{i=1}^n \subset X^0$. Let $K_i$ be the subgroup generated by the horizontal closed paths based at $C_i$ for each $i$. Fix a $0$-cube $C_i$. Then $K_i$ is finitely generated since $X$ is compact. Moreover, since horizontal paths immerse in the underlying graph, the map $X\rightarrow \Gamma_X$ induces an injective homomorphism $K_i\rightarrow  \pi_1\Gamma_X$. Identify $K_i$ with its image. Let $\left\{\gamma_{ij}\right\}_{j=1}^m$ be the set of all embedded non-closed horizontal paths between $0$-cubes $C_i$ and $C_j$. Each $\gamma_{ij}$ maps to an essential closed path in $\Gamma_X$ and thus represents a nontrivial element $w_{ij}\in\pi_1\Gamma_X$. Furthermore, $w_{ij}\notin K_i$. Indeed, since the attaching maps of edge-spaces are injective, the horizontal graphs immerse in $\Gamma_X$, and so the elements represented by $\gamma_{ij}$ are distinct from elements of $K_i$. \textcolor{black}{In particular, the products of finitely many cosets $K_iw_{ii_1}K_{i_1}w_{i_{1}i_{2}}K_{i_2}\cdots$ does not contain the identity element. Note that there are finitely many such products of cosets.} By Theorem \ref{thm:RZ}, there exists a finite index normal subgroup \textcolor{black}{$N\unlhd \pi_1\Gamma_X$ that is disjoint from all such multiple cosets.} Let $p:\overline{X}\rightarrow X$ be the covering map corresponding to $N$. Let $Z\subset \overline{X}$ be a vertex-space and $\overline{C_i},  \overline{C_j}\in Z$ be $0$-cubes mapping to $0$-cubes $C_i, C_j\in X$. Then, $\overline{C_i}$ and $\overline{C_j}$ are not in the same E-parallelism class of $\overline{X}$. Indeed, if $\overline{\gamma}$ is a horizontal path in $\overline{X}$ \textcolor{black}{from} $\overline{C_i}$ to $\overline{C_j}$, then $p\left(\overline{\gamma}\right)$ is a horizontal path $\gamma$ in $X$ which represents an element in
$K_iw_{ii_1}K_{i_1}w_{i_{1}i_{2}}K_{i_2}\cdots w_{i_{n}j}K_j$, \textcolor{black}{where $w_{ii_{1}}, w_{i_{1}i_{2}},\ldots,w_{i_{n}j}$ are the elements of $\pi_1\Gamma_X$ representing non closed embedded paths between $0$-cubes $C_{i}, C_{i_1},C_{i_2}\cdots,C_{j}$, respectively}. However, $N$ contains no such elements, and thus $q:\overline{X}\rightarrow \overline{X}^{E}$ is strict.

\textcolor{black}{Finally, any normal finite index subgroup $N'\subset N$ induces a finite cover $\overline{X}'\rightarrow\overline{X}\rightarrow X$ with the same properties as $\overline{X}$. That is, $\overline{X}'$ splits as a graph of spaces with vertex-spaces isomorphic to $Y$ and the horizontal quotient $\overline{X}'\rightarrow {\overline{X}'}^{E}$ is strict. } 
\end{proof}
\begin{lem}\label{lem:013}
Let $X\rightarrow \Gamma_X$ be a graph of nonpositively curved cube complexes and $q:X\rightarrow X^{E}$ be a strict horizontal quotient, where $X^{E}$ is nonpositively curved. Let $Y$ be a  vertex-space of $X$. If $X$ has no inter-osculating hyperplanes, then  $q\left(Y\right)\subset X^{E}$ is a locally convex subcomplex. 
\end{lem}
\begin{proof}
It suffices to show that $q\left(Y\right)$ has no missing squares in $X^{E}$. To do so, we show that for each $0$-cube $y\in q\left(Y\right)$, the inclusion $\link_{q\left(Y\right)}\left(y\right)\subset \link_{X^{E}}\left(y\right)$ is full.

Let $y\in q\left(Y\right)$ be a $0$-cube, and let $e\in \link_{X^{E}}\left(y\right)$ be a $1$-simplex whose boundary $0$-simplices $x_1$ and $x_2$ lie in $\link_{q\left(Y\right)}\left(y\right)$ with $e\notin \link_{q\left(Y\right)}\left(y\right)$. Since $q$ is strict, there are consecutive $1$-cubes $a_1, a_2\in q\left(Y\right)$ containing $y$ that are identified with consecutive $1$-cubes of a $2$-cube $S_e\not\subset q\left(Y\right)$. Since $X^{E}$ is nonpositively curved, $e$ is the only $1$-simplex containing $x_1$ and $x_2$ and so $a_1$ and $a_2$ are not consecutive $1$-cubes of a $2$-cube in $q\left(Y\right)$. Then $X$ contains inter-osculating hyperplanes $U_1, U_2$ which cross in $S_e'\subset q^{-1}\left(S_e\right)$, and are dual to $a_1'\subset q^{-1}\left(a_1\right)$ and $a_2'\subset q^{-1}\left(a_2\right)$, respectively, where $a_1', a_2'$ share a common $0$-cube $y'\in q^{-1}\left(y\right)$ but don't bound a corner of a $2$-cube in $Y$. A contradiction. See Figure~\ref{2000.png}.
\begin{figure}[t]\centering
\includegraphics[width=.6\textwidth]{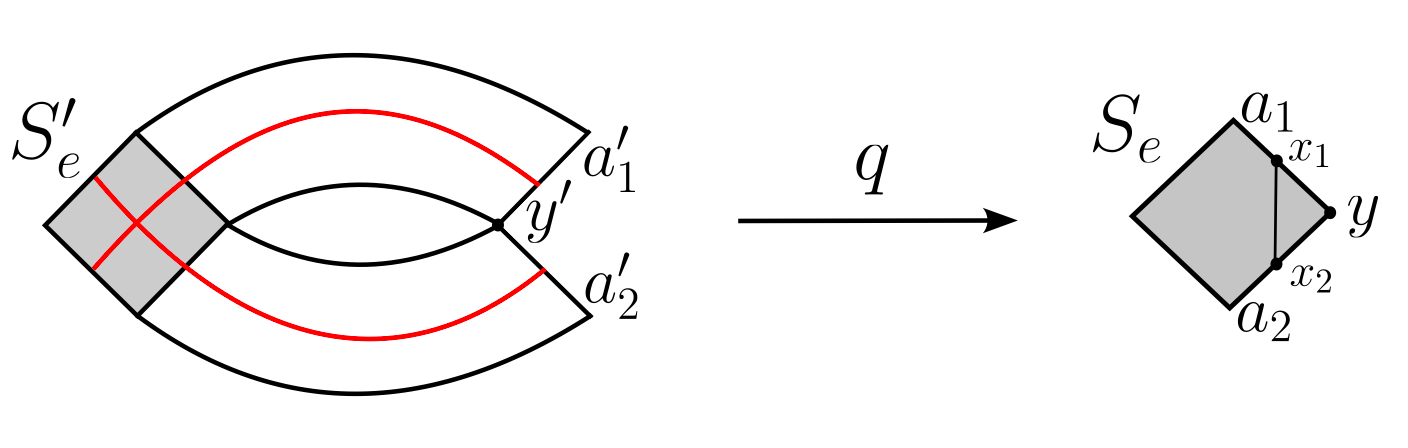}
\caption[]{Inter-osculation arising from consecutive $1$-cubes not bounding a $2$-cube in $Y$.}{\label{2000.png}
}
\end{figure}
\end{proof}

The strategy for obtaining $\overline{X}^{E}$ that is special, is to use multiple coset separability properties of $F$ acting on $\widehat X$ to obtain a compact special cube complex $\overline{X}$ whose horizontal quotient $\overline{X}^{E}$ is special. The property that hyperplanes are embedded and 2-sided is preserved under the map $\overline{X} \rightarrow \overline{X}^{E}$. However, non-inter-osculation and non-self-osculation are not necessarily preserved by $\overline{X}\rightarrow \overline{X}^{E}$. We are therefore forced to revisit and prove a more powerful form  of Theorem~\ref{hw2} , that provides an intermediate cover $\overline{X}$ for which $\overline{X}^{E}$ retains all desired properties.
\begin{lem}\label{no remote osculation}
Let $\widehat{X}\rightarrow \Gamma_ {\widehat{X}}$ be a controlled tree of compact nonpositively curved cube complexes \textcolor{black}{with isomorphic vertex-spaces}. Let $F$ be a free group acting freely and cocompactly on $\Gamma_ {\widehat{X}}$ and $\widehat{X}$, so that $\widehat{X}\rightarrow \Gamma_ {\widehat{X}}$ is $F$-equivariant. Suppose $\widehat{X}$ \textcolor{black}{is special}. Then there is a finite index normal subgroup $N\subset F$ and a covering map $\widehat{X}\rightarrow N\backslash\widehat{X}=\overline{X}$ where $\overline{X}$ splits as a graph of cube complexes whose horizontal quotient $\overline{X}^{E}$ contains no self-osculating hyperplanes and no inter-osculating hyperplanes.
\end{lem}
\begin{proof}
Since $\widehat{X}$ has no self-crossing hyperplanes, we can identify each immersed hyperplane with its image in $\widehat{X}$.  We first find a finite graph of cube complexes $\overline{X}$ whose horizontal quotient has no inter-osculating hyperplanes. We do so by finding an appropriate finite index subgroup $N\subset F$ and taking the quotient $N\backslash\widehat{X}=\overline{X}$. \textcolor{black}{Note that Lemma~\ref{lem:012} allows us to pass to a finite cover, if necessary, to ensure that the  horizontal quotient is a cube complex.} By Lemma~\ref{lem:007}, the horizontal quotient $\overline{X}^{E}$ has inter-osculation if $\overline{X}$ has remote inter-osculation. Remote inter-osculation in $\overline{X}$ occurs if there are crossing hyperplanes $A, B$ of $\widehat{X}$ and an element $g\in F$ such that $gB$ and  $A$ osculate with a horizontal graph $T$ in $\widehat{X}$. Such an element is called a \textit{remote inter-osculator} at $T$. Let $\mathcal{R}\subset F$ be the set of remote inter-osculators. We characterize the elements of $\mathcal{R}$ and use subgroup separability to find a finite index subgroup of $F$ that is disjoint from the set $\mathcal{R}$.

By $F$-cocompactness, there are finitely many $F$-orbits of horizontal graphs. Let $\left\{T_i\right\}_{i=1}^m$ be their representatives. For each tree $T_i\in \left\{T_i\right\}_{i=1}^m$ there are finitely many $\Stab\left(T_i\right)$-orbits of hyperplanes that osculate with $T_i$. Let $\left\{A_{ij}\right\}_{j=1}^{r_i}$ be their representatives. Similarly, for each hyperplane $A_{ij}\in \left\{A_{ij}\right\}_{j=1}^{r_i}$, there are finitely many $\Stab\left(A_{ij}\right)$-orbits of hyperplanes crossing $A_{ij}$. Let $\left\{B_{ijk}\right\}_{k=1}^{s_{ij}}$ be their representatives. See Figure~\ref{fig:5980}. 

For each $B_{ijk}$ and $A_{ir}$, if there is an element $h_{ijkr}$ mapping $B_{ijkr}$ to $A_{ir}$, then the set of all elements $g$ with $gB_{ijk}=A_{ir}$ is:
$$\mathcal{O}_{ijkr}=\stab\left(A_{ir}\right)h_{ijkr}\stab\left(B_{ijk}\right)$$

\begin{figure}[t]\centering
\includegraphics[width=.5\textwidth]{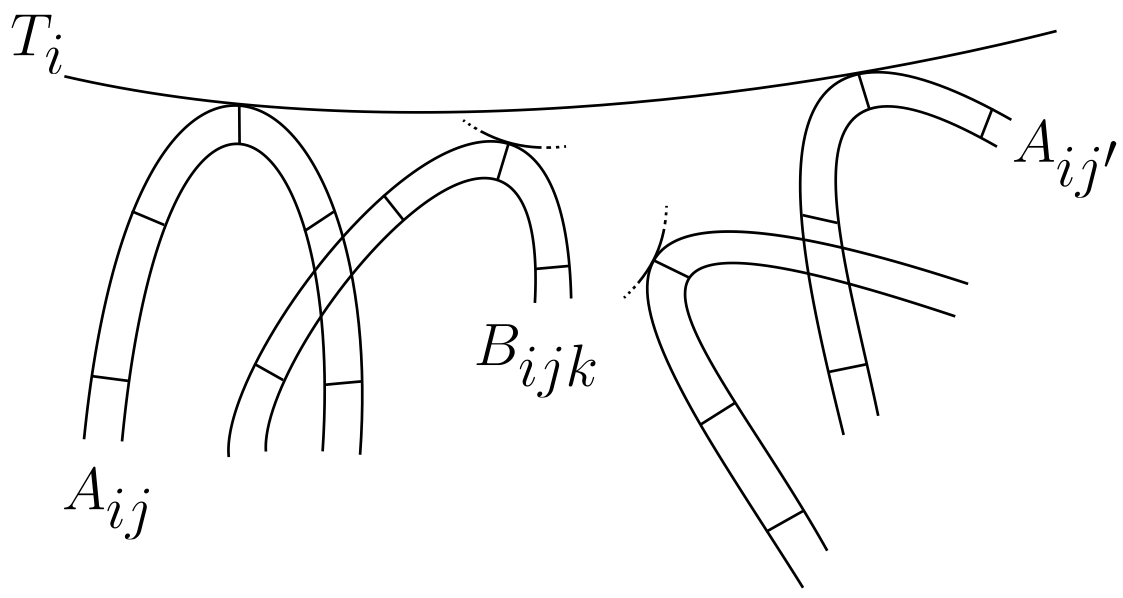}
\caption[$k$-corners ]{\label{fig:5980}
The hyperplane $B_{ijk}$ crosses $A_{ij}$ which osculate with the horizontal graph $T_i$. The element $g$ maps $B_{ijk}$ to $A_{ir}$ which also osculates with $T_i$.}
\end{figure}
Furthermore, by precomposing $g\in\mathcal{O}_{ijkr}$ with elements of $\stab\left(A_{ij}\right)\stab\left(T_i\right)$, postcomposing $g$ with elements of  $\stab\left(T_i\right)$, and then taking the union over $j, k, r$, we obtain the set of remote inter-osculators at $T_i$: $$\mathcal{O}_i=\displaystyle\bigcup_{jkr}\stab\left(T_i\right)\stab\left(A_{ir}\right)h_{ijkr}\stab\left(B_{ijk}\right)\stab\left(A_{ij}\right)\stab\left(T_i\right)$$

Let $\mathcal{O}=\displaystyle\bigcup_{i}\mathcal{O}_i$. Each horizontal graph $T$ is a translate of some $T_i$. Thus each remote inter-osculator at $T$ is conjugate to an element of $\mathcal{O}$. By assumption, $\widehat{X}$ contains no inter-osculating hyperplanes. By Lemma~\ref{lem:011}, $\widehat{X}$ has no remote inter-osculation and thus, $1_F\notin \mathcal{O}$. By cocompactness, the stabilizers are finitely generated. By Theorem~\ref{thm:RZ}, the set $\mathcal{O}$ is closed in the profinite topology, and so there exists a finite index normal subgroup $N$ disjoint from $\mathcal{O}$, and hence disjoint from $\mathcal{R}$. Then the horizontal quotient of $N\backslash\widehat{X}\rightarrow \left(N\backslash\widehat{X}\right)^{E}$ has no inter-osculating hyperplanes.

Similarly, to find $\overline{X}\rightarrow \overline{X}^{E}$ with no self-osculating hyperplanes, we use the same method and follow the steps sketched below.

An element $g\in F$ gives rise to self-osculation in $\overline{X}^{E}$ if $gA=A'$ where $A$ and $A'$ are hyperplanes osculating with the same horizontal graph $T$. Such elements are called \textit{remote self-osculators} at $T$. The set of remote self-osculators at $T_i$ is: 
$$\mathcal{S}_i=\displaystyle\bigcup_{jr}\stab\left(T_i\right)\stab\left(A_{ir}\right)h_{ijr}\stab\left(A_{i{j}}\right)\stab\left(T_i\right)$$
Then any remote self-osculator is conjugate to an element of $\mathcal{S}=\displaystyle\bigcup_{i}\mathcal{S}_i$. By Lemma~\ref{lem:011}, we have $1_F\notin \mathcal{S}$. Then there exists a finite index normal subgroup $N'\subset F$ such that $N'\backslash\widehat{X}\rightarrow \left(N'\backslash\widehat{X}\right)^{E}$ has no self-osculating hyperplanes and the following diagram commutes:

\begin{center}
\begin{tikzcd}
\widehat{X} \arrow[r] \arrow[d] & \Gamma_ {\widehat{X}} \arrow[d] \\
\overline{X} \arrow[r]                      & \Gamma_{\overline{X}}             
\end{tikzcd}
\end{center}

The map $\widehat{X}\rightarrow \left(N\cap N'\right)\backslash\widehat{X}=\overline{X}$ provides the desired covering map.
\end{proof}
\begin{rem}\label{rem:special}
By taking double covers, if necessary, we can ensure that the hyperplanes in $\overline{X}$ are two-sided, which, by Lemma~\ref{lem:003}, means that the hyperplanes of $\overline{X}^{E}$ are two sided as well. \textcolor{black}{Moreover, any finite cover induced by a finite index subgroup of $N\cap N'$ has the properties stated in Lemma~\ref{no remote osculation}.}
\end{rem}

\textcolor{black}{Up until this point, we have shown how to ﬁnd a compact quotient where the pathologies precluding specialness do not appear in the horizontal quotients. In the remainder of this section, we show how to ensure that the horizontal quotient is nonpositively curved.}
\begin{defn}[$k$-corners]
For $k\in \left\{1,2,3\right\}$, a \textit{$k$-cycle of squares} is a planar complex $S_k$ formed by gluing $k$ squares around a vertex $v$. A $k$-cycle of squares has $k$ hyperplanes $\left\{\alpha_i\mid 1\leq i\leq k\right\}$ and $k$ codimension-$2$ hyperplanes $\left\{\beta_j\mid 1\leq j\leq k\right\}$. \textcolor{black}{Recall that a codimension-$2$ hyperplane is the 
intersection of two distinct pairwise intersecting hyperplanes, and the carrier of a codimension-$2$ hyperplane is the cubical neighborhood containing the intersection.} See Figure~\ref{fig:2980}.
\begin{figure}[t]\centering
\includegraphics[width=.7\textwidth]{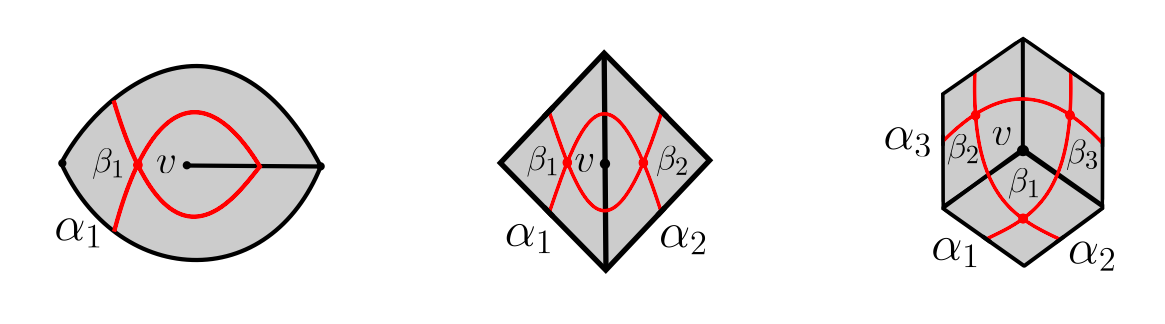}
\caption[$k$-corners ]{\label{fig:2980}
$1$-cycle, $2$-cycle, and $3$-cycle of squares with their dual curves.}
\end{figure}

Let $X$ be a cube complex and $D\subset X$ be an $n$-cube. An \textcolor{black}{$\left(n+2\right)$-dimensional} \textit{$k$-corner of $X$ at $D$} is a combinatorial immersion $\left(Z_k, I^n\right)\rightarrow \left(X,D\right)$ where $Z_k=S_k\times I^n$ and $I^n$ is identified with $\left\{v\right\}\times I^n$ in $Z_k$. We write $Z_k\rightarrow X$ when the map $I^n\rightarrow D$ is clear from the context.  

A $k$-corner is \textit{empty} if $\left(Z_k, I^n\right)\rightarrow \left(X,D\right)$ does not extend to \textcolor{black}{ $\left(I^{n+3},I^n\right)\rightarrow \left(X,D\right)$}. Note that $1$-corners and $2$-corners are always empty. Furthermore, under the immersion $Z_k\rightarrow X$, hyperplanes map to hyperplanes and crossing hyperplanes map to crossing hyperplanes. 
\end{defn}
\begin{rem}\label{rem:004} Nonpositive curvature can be expressed in terms of $k$-corners. Specifically, a cube complex is nonpositively curved if it has no empty $k$-corners. Indeed, if $\link_X\left(D\right)$ has a loop [a bigon] then $X$ has a $1$-corner [a $2$-corner] at $D$. Furthermore, if the no-$\triangle$ property fails at $D$, then $X$ has an empty $3$-corner at $D$.

We also note that if $X$ has an empty $k$-corner at $D$, then $\link_X\left(x\right)$ is not flag for each $0$-cube $x$ of $D$.
\end{rem}
\begin{defn}[$k$-precorners]
Let $X\rightarrow \Gamma_X$ be a graph of cube complexes with cubical horizontal quotient $q: X\rightarrow X^{E}$. Let $Z_k\xrightarrow{\varphi} X^{E}$ be an $\left(n+2\right)$-dimensional $k$-corner and let $\left\{A_i=\alpha_i\times I^n\mid 1\leq i\leq k\right\}$ be hyperplanes of $Z_k=S_k\times I^n$ where $\left\{\alpha_i\mid 1\leq i\leq k\right\}$ are the hyperplanes of $S_k$. Let $\left\{B_j=\beta_j\times I^n\mid 1\leq j\leq k\right\}$ be codim\-ension-$2$ hyperplanes of $Z_k=S_k\times I^n$ where $\left\{\beta_j\mid 1\leq j\leq k\right\}$ are the codimension-$2$ hyperplanes of $S_k$. Let $\left\{H_i\xrightarrow{h_i} X\mid 1\leq i\leq k\right\}$ be the immersed hyperplanes of $X$ such that $\varphi\left(A_i\right)\subset \left(q\circ h_i\right)\left(H_i\right)$, and let $N\left(H_i\right)\rightarrow X$ be their immersed carriers.

The $\left(n+2\right)$-dimensional $k$-\textit{precorner} $P_k$ over the $\left(n+2\right)$-dimensional $k$-corner $Z_k$ is the disjoint union of the corresponding immersed carriers $N\left(H_i\right)\rightarrow X$ amalgamated along the carriers of \textcolor{black}{the codimension-$2$ hyperplanes of $H_i$ that contain the preimages}  $h_i^{-1}\left(q^{-1}\left(B_j\right)\right)$. See Figure~\ref{fig:3980}. Note that there is a \textit{global} \textcolor{black}{map} $h:P_k\rightarrow X$ that restricts to $h_i$ on each immersed hyperplane $H_i$.

A $k$-precorner $P_k\xrightarrow{h}X$ over a $k$-corner $Z_k\xrightarrow{\varphi} X^{E}$ is \textcolor{black}{\textit{empty} if $Z_k\xrightarrow{\varphi} X^{E}$ is empty. $P_k\xrightarrow{h} X$} is \textit{trivial} if 
$\varphi$ lifts to a combinatorial map $Z_k\xrightarrow{} X$ such that the following diagram commutes:

\begin{center}
\begin{tikzcd}
P_k \arrow[r, "h"]                                      & X \arrow[d, "q"]                               \\
Z_k \arrow[r, "\varphi"] \arrow[ru] & X^{E}
\end{tikzcd}
\end{center}
\textcolor{black}{}
\end{defn}

\begin{rem}\label{rem:005}
The \textcolor{black}{map} $P_k\xrightarrow{h} X$ induces a splitting of $P_k$ as a graph of spaces as in the following commutative diagram:

\begin{center}
\begin{tikzcd}
P_k \arrow[d, "h"'] \arrow[r] & \Gamma_{P_k} \arrow[d] \\
X \arrow[r]                   & \Gamma_X              
\end{tikzcd}
\end{center}
Specifically, the vertex-spaces of $P_k$ are the components of the preimages of vertex-spaces of $X$ and the edge-spaces of $P_k$ are the components of the preimages of edge-spaces of $X$. The graph $\Gamma_{P_k}$ is the quotient of $P_k$ obtained by identifying vertex-spaces and edge-spaces of $P_k$ with vertices and edges of $\Gamma_{P_k}$, respectively. The composition $P_k\rightarrow X\rightarrow \Gamma_X$ induces a graph morphism $\Gamma_{P_k}\rightarrow \Gamma_X$ that maps vertices to vertices and open edges to open edges. \textcolor{black}{Note that when a $k$-precorner $P_k\rightarrow X\rightarrow X^{E}$ over a $k$-corner $Z_k\xrightarrow{\varphi} X^{E}$ is trivial, the lift of $\varphi$ maps $Z_k$ into a vertex-space of $X$. This induces a map $Z_k\rightarrow P_k$ whose range lies in a vertex-space of $P_k$ such that the following diagram commutes:}
\begin{center}
    \begin{tikzcd}
P_k \arrow[r, "h"]                                    & X \arrow[d, "q"] \\
Z_k \arrow[r, "\varphi"] \arrow[u, dashed] \arrow[ru] & X^{E}           
\end{tikzcd}
\end{center}
\end{rem}

\begin{lem}\label{lem:015}
Let $\widehat{X}\rightarrow \Gamma_{\widehat{X}}$ be a tree of \textcolor{black}{nonpositively curved} cube complexes where the attaching maps of edge-spaces are injective local isometries. Let $\widehat{X}\rightarrow \widehat{X}^{E}$ be the horizontal quotient and let $P_k\rightarrow \widehat{X}$ be a $k$-precorner over a $k$-corner $Z_k\xrightarrow{\varphi} \widehat{X}^{E}$. Then $P_k$ is trivial \textcolor{black}{and hence nonempty}. 
\end{lem}
\begin{proof} 
Let $T\subset \widehat{X}$ be a minimal connected subtree of spaces containing $k$ cubes $\left\{C_i\subset P_k\right\}_{i=1}^k$ that map onto $\varphi\left(Z_k\right)$. \textcolor{black}{Then $T$ is finite since any} $k$ cubes mapping onto $\varphi\left(Z_k\right)$ must lie in a finite \textcolor{black}{connected} subcomplex of $\widehat{X}$. Note that the minimality is under inclusion and over all possible collections of $k$ cubes mapping onto $\varphi\left(Z_k\right)$.  Let $T\rightarrow \Gamma_T$ be the underlying tree. We claim that $\Gamma_T$ is a vertex. \textcolor{black}{Note that if $k=1$ then there is only one cube that lies in a single vertex-space which by the minimality of $T$, implies that $\Gamma_T$ is a vertex. So we can assume $2\leq k\leq 3$.} Suppose that $\Gamma_T$ has a spur $e$ incident on vertices $v_1$ and $v_2$, where $\deg\left(v_1\right)=1$.  Let $T_e$ be the corresponding edge-space attached to the vertex-spaces $T_{v_1}$ and $T_{v_2}$. By the minimality of $T$, we can assume without loss of generality that $T_{v_1}$ contains exactly one cube $C_i$. There exist \textcolor{black}{distinct} immersed hyperplanes $H_1\rightarrow \widehat{X}$ and $H_2\rightarrow \widehat{X}$ that cross in $C_i$ and extend to $T_{v_2}$ through $T_e$. Since the attaching maps are local isometries, $C_i$ must be in the edge-space. But in that case, the edge-space $T_e \times \left[-1,1\right]$ contains $C_i\times \left[-1,1\right]$ and so the vertex-space $T_{v_2}$ contains $C_i\times \left\{-1\right\}$. Therefore, there exists a proper subtree $T'\subset T$ containing $k$ cubes mapping onto $\varphi\left(Z_k\right)$, contradicting the minimality of $T$.

Since $\Gamma_T$ is finite and has no spurs, $T$ is a vertex-space. Moreover, $\widehat{X}$ is a tree of spaces, and so the restriction  $q|_{_T}:T\rightarrow \widehat{X}^{E}$ is an embedding. \textcolor{black}{This provides the required map $Z_k\rightarrow T\subset \widehat{X}$. So $P_k$ is trivial. By assumption, the vertex-spaces of $\widehat{X}$ are nonpositively curved. By Remark~\ref{rem:004}, $Z_k$ (and hence $P_k$) is a nonempty $k$-corner ($k$-precorner)}. 
\end{proof}

\begin{figure}[t]\centering
\includegraphics[width=.6\textwidth]{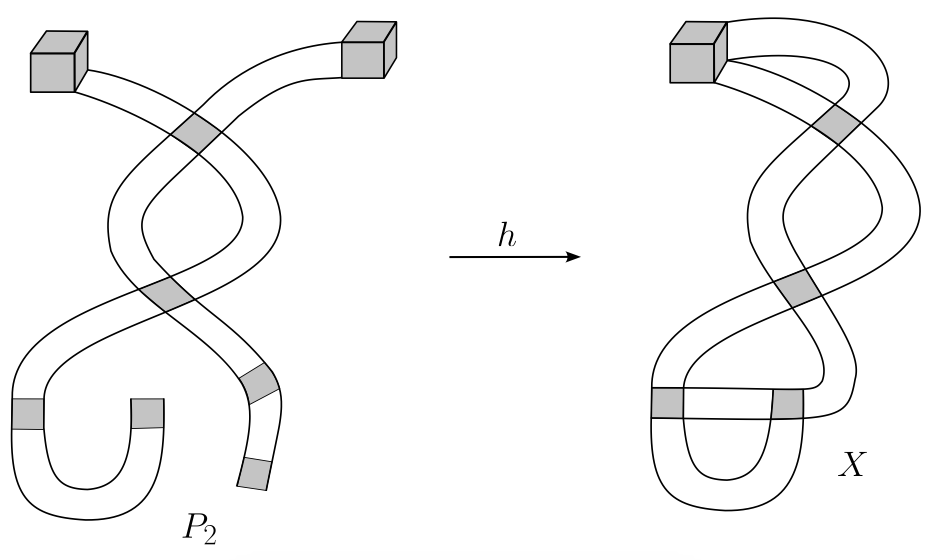}
\caption[$k$-corners ]{\label{fig:3980}
A $2$-precorner.}
\end{figure}

\begin{defn}\label{defn:kchains} Let \textcolor{black}{$X\rightarrow \Gamma_X$ be a graph of cube complexes} and let $F$ be a group acting on $X$ \textcolor{black}{and $\Gamma_X$ so that the map $X\rightarrow \Gamma_X$ is $F$-equivariant}. Given $k\in\left\{1,2,3\right\}$, a \textit{$k$-chain} is an ordered $\left(k+1\right)$-tuple of immersed hyperplanes $\left(H_t\right)_{t=0}^k$ where $H_{t-1}$ crosses $H_t$ for all $1\leq t\leq k$. See Figure~\ref{fig:4980}. 

\textcolor{black}{Given a subgroup $G\subset F$, an element $g\in G\subset F$ is a \textit{closing element} if $g$ maps $H_k$ in some $k$-chain $\left(H_t\right)_{t=0}^k$ to $H_0$ so that under the maps 
$X\rightarrow G\setminus X\rightarrow \left(G\setminus X\right)^{E}$, the hyperplanes $H_t$, for $0\leq t\leq k$, are mapped to the hyperplanes of an empty $k$-corner in $\left(G\setminus X\right)^{E}$}. We say $\left(H_t\right)_{t=0}^k$ is \textit{closed} by $g$. 
\begin{figure}[t]\centering
\includegraphics[width=.7\textwidth]{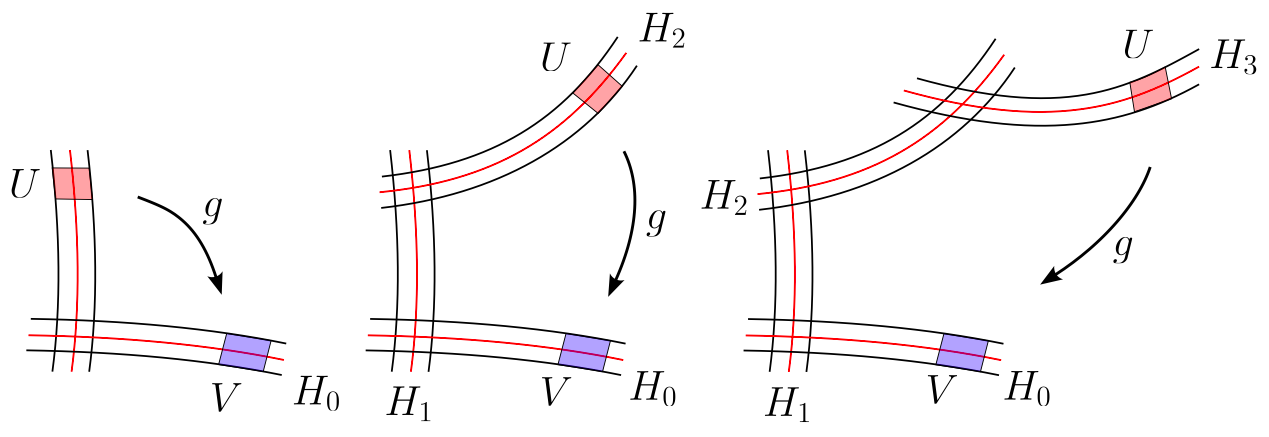}
\caption[$k$-corners ]{\label{fig:4980}
From left: $1$-chain, $2$-chain, and $3$-chain.}
\end{figure}
\end{defn}
\begin{rem}\label{rem:closing}
\textcolor{black}{Let $\widehat{X}$ be a tree of compact isomorphic cube complexes and let $F$ be a group acting freely and cocompactly on $\widehat{X}$. If for some subgroup $G\subset F$, the quotient $G\setminus \widehat{X}$ has an \textcolor{black}{empty} $k$-precorner, then $G$ contains a closing element of some $k$-chain in $\widehat{X}$.
Note that closing elements map codimension-$2$ hyperplanes to codimension-$2$ hyperplanes.} 
\end{rem}
\begin{defn}\label{rem:007} \textcolor{black}{Let $B$ be a compact bouquet of circles and let $X\rightarrow B$ be a graph of cube complexes with one compact vertex-space. Let $\widehat{X}\rightarrow X$ be the covering map induced by the universal covering map $\widetilde{B}\rightarrow B$ so that the following diagram commutes:}
\begin{center}
    \begin{tikzcd}
\widehat{X} \arrow[d] \arrow[r] & \Gamma_{\widehat{X}}=\widetilde{B} \arrow[d] \\
X \arrow[r]                     & B                                           
\end{tikzcd}
\end{center}
Then the free cocompact action $F\curvearrowright \widetilde{B}$ \textcolor{black}{extends to} a free cocompact action $F\curvearrowright \widehat{X}$ \textcolor{black}{mapping vertex-spaces to vertex-spaces}. \textcolor{black}{Below, we consider only the hyperplanes of $\widehat{X}$ that do not lie in the middle of thick edge-spaces. That is, hyperplanes distinct from $X_e\times \left\{0\right\}\subset X_e\times [-1,1]$ for some edge-space $X_e$.} We fix a finite collection of \textcolor{black}{such} immersed hyperplanes $L=L_0\cup L_1\cup L_2\cup L_3$ where:\\
$L_0=\left\{H_1,\ldots,H_{n_0}\right\}$ are $F$-representatives of hyperplanes;\\
$L_1=\displaystyle\bigcup_{i} \left\{H_{i1},\ldots,H_{i{n_i}}\right\}$ are $\Stab\left(H_i\right)$-representatives of hyperplanes crossing $H_i$, for $1\leq i\leq \textcolor{black}{n_0}$;\\
$L_2=\displaystyle\bigcup_{i,j} \left\{H_{ij1},\ldots,H_{ij{n_{ij}}}\right\}$ are $\Stab\left(H_{ij}\right)$-representatives of hyperplanes crossing $H_{ij}$, for $1\leq i\leq \textcolor{black}{n_0}$ and $1\leq j\leq \textcolor{black}{n_i}$; and\\
$L_3=\displaystyle\bigcup_{i,j,t} \left\{H_{ijt1},\ldots,H_{ijt{n_{ijt}}}\right\}$ are $\Stab\left(H_{ijt}\right)$-representatives of hyperplanes crossing $H_{ijt}$, for $1\leq i\leq \textcolor{black}{n_0},\ 1\leq j\leq \textcolor{black}{n_i}$, and $1\leq t\leq \textcolor{black}{n_{ij}}$.

\textcolor{black}{Let $\mathcal{C}$ be the set of all $k$-chains of hyperplanes of $L$. For each hyperplane $A\in L$, there are finitely many $\Stab\left(A\right)$-representatives of codimension-$2$ hyperplanes in $N\left(A\right)$. Choose such representatives. For each $k$-chain  $C=\left(A_t\right)_{t=0}^k$, with $A_t\in L_t$, let $J_C$ be the set of elements of $F$ that map \textcolor{black}{the chosen} $\stab\left(A_k\right)$-representatives of codimension-$2$ hyperplanes of $A_k$ to \textcolor{black}{the chosen} $\stab\left(A_0\right)$-representatives of codimen\-sion-$2$ hyperplanes of $A_0$. 
Note that $J_C$ is finite.}
\end{defn}
\begin{lem}\label{lem:016}
\textcolor{black}{Let $B$ be a compact bouquet of circles and let $X\rightarrow B$ be a graph of cube complexes with one compact vertex-space. Let $\widehat{X}\rightarrow \Gamma_{\widehat{X}}=\widetilde{B}$ be the tree of cube complexes where $\widetilde{B}\rightarrow B$ is the universal covering map such that the following diagram commutes: }
\begin{center}
\begin{tikzcd}
\widehat{X} \arrow[r] \arrow[d] & \Gamma_{\widehat{X}}=\widetilde{B} \arrow[d] \\
X \arrow[r]                      & B                      
\end{tikzcd}
\end{center}
Let $F=\pi_1B$ be the free group acting freely and cocompactly on $\widehat{X}$ and let $g\in F$ be a closing element. Then there exist a $k$-chain $C=\left(A_t\right)_{t=0}^k$ with $A_t\in L_t$, and $f\in F$ such that
\begin{center}
$f^{-1}gf\in \stab\left(A_0\right)J_C\stab\left(A_k\right)\Stab\left(A_{k-1}\right)\cdots\stab\left(A_0\right)$
\end{center}
\end{lem}

\begin{proof} 
Let $\left(B_t\right)_{t=0}^k$ be the $k$-chain in $\widehat{X}$ closed by $g$. Let $V$ and $U$ be codimension-$2$ hyperplanes in $N\left(B_0\right)$ and $N\left(B_k\right)$, respectively, such that $gU=V$. The hyperplane $B_0$ lies in the $F$-orbit of some representative $A_0\in L_0$ and so $B_0=fA_0$ for some $f\in F$. Let $\left\{V_s\mid 1\leq s\leq n_0\right\}$ be $\Stab\left(A_0\right)$-representatives of orbits of codimension-$2$ hyperplanes in $N\left(A_0\right)$. Let $a_0\in\Stab\left(A_0\right)$ and $V_s\in \left\{V_s\mid 1\leq s\leq n_0\right\}$ such that $V=fa_0V_s$. Let $A_1\in L_1$ and $a_0'\in \Stab\left(A_0\right)$ such that $B_1=fa_0'A_1$.

\underline{\textit{Case $k=1$}}: Let $\left\{U_r\mid 1\leq r\leq n_1\right\}$ be $\Stab\left(A_1\right)$-representatives of orbits of codim\-ension-$2$ hyperplanes in $N\left(A_1\right)$. Then $U=fa_0'a_1 U_r$ for some $a_1\in \Stab\left(A_1\right)$ and $U_r\in \left\{U_r\mid 1\leq r\leq n_1\right\}$. So $$gU=V\ \Rightarrow\ gfa_0'a_1U_r=fa_0V_s\ \Rightarrow \left(a_0^{-1}f^{-1}gfa_0'a_1\right)U_r=V_s$$
Therefore $\left(a_0^{-1}f^{-1}gfa_0'a_1\right)\in J_C$, for $C=\left(A_t\right)_{t=0}^1$, and so
$$f^{-1}gf\in \stab\left(A_0\right)J_C\stab\left(A_1\right)\stab\left(A_0\right)$$

\underline{\textit{Case $k=2$}}: We have $B_2=fa_0'a_1A_2$ for some $A_1\in L_1$ and $a_1\in \Stab\left(A_1\right)$. Then $U=fa_0'a_1a_2U_r$ where $a_2\in \Stab\left(A_2\right)$ and $U_r$ is a $\Stab\left(A_2\right)$-representative in $\left\{U_r\mid 1\leq r\leq n_1\right\}$. So, $gU=V\ \Rightarrow \ g\left(fa_0'a_1a_2\right)U_r=fa_0V_s$. Therefore,
$$f^{-1}gf\in \Stab\left(A_0\right)J_C\Stab\left(A_2\right)\Stab\left(A_1\right)\Stab\left(A_0\right)$$

\underline{\textit{Case $k=3$}}: Similarly, $U=fa_0'a_1a_2a_3U_r$ where $a_3\in \Stab\left(A_3\right)$ and $U_r$ is a $\Stab\left(A_3\right)$-representative in $\left\{U_r\mid 1\leq r\leq n_1\right\}$. Thus, $$(gU=V)\Rightarrow g\left(fa_0'a_1a_2a_3\right)U_r=fa_0V_s$$ and so
\begin{center}
$f^{-1}gf\in \Stab\left(A_0\right)J_C\Stab\left(A_3\right)\Stab\left(A_2\right)\Stab\left(A_1\right)\Stab\left(A_0\right)$
\end{center}
See Figure~\ref{2022} for case $k=2$.
\begin{figure}[t]\centering
\includegraphics[width=.7\textwidth]{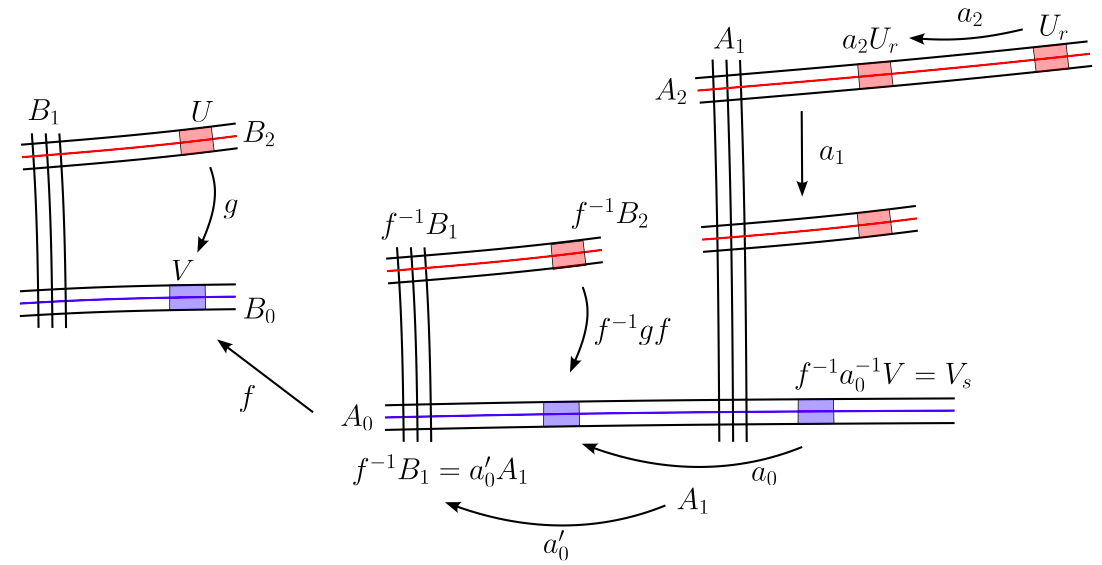}
\caption[$k$-corners ]{\label{2022}
Case $k=2$}
\end{figure}
\end{proof}
\begin{lem}\label{lem:017}
\textcolor{black}{Let $B$ be a compact bouquet of circles and let $X\rightarrow B$ be a graph of cube complexes with one compact nonpositively curved vertex-space and embedded locally convex edge-spaces. Let $\widehat{X}\rightarrow X$ be the covering map induced by the universal covering map $B\rightarrow \widetilde{B}$ where $\widehat{X}$ is a tree of compact nonpositively curved cubes complexes and the following diagram commutes: }
\begin{center}
\begin{tikzcd}
\widehat{X} \arrow[r] \arrow[d] & \Gamma_{\widehat{X}}=\widetilde{B} \arrow[d] \\
X \arrow[r]                      & B                      
\end{tikzcd}
\end{center}
Let $F=\pi_1B$ be the free group acting freely and cocompactly on $\Gamma_{\widehat{X}}\textcolor{black}{\ =\widetilde{B}}$ inducing a free cocompact $F$-action on $\widehat{X}$. Then there exists a compact graph of cube complexes $\overline{X}\rightarrow \Gamma_{\overline{X}}$ and a regular covering map $\widehat{X}\rightarrow \overline{X}$ such that the following diagram commutes and the horizontal quotient $\overline{X}\rightarrow \overline{X}^{E}$ is nonpositively curved:
\begin{center}
\begin{tikzcd}
\widehat{X} \arrow[r] \arrow[d] & \Gamma_{\widehat{X}} \arrow[d] \\
\overline{X} \arrow[r]                      & \Gamma_{\overline{X}}                      
\end{tikzcd}
\end{center}
\textcolor{black}{Furthermore, any intermediate covering map $\widehat{X}\rightarrow \overline{X}'\rightarrow \overline{X}$ induced by a finite index normal subgroup of $ \pi_1\Gamma_{\overline{X}}$ splits as a graph of cube complexes with nonpositively curved horizontal quotient.}
\end{lem}
\begin{proof}
\textcolor{black}{Using Lemma~\ref{lem:012}, we can ensure that any finite cover $\overline{X}$ we find below admits a cubical horizontal quotient.} Fix collections $L$ and  $\mathcal{C}$ as in Definition~\ref{rem:007}. Let \textcolor{black}{$$\mathcal{O}=\displaystyle\bigcup_{1\leq k\leq 3}\ \displaystyle\bigcup_{C\in \mathcal{C}}\left(\stab\left(A_0\right)J_C\stab\left(A_k\right)\cdots\stab\left(A_0\right)\right)$$} 
where $C=\left(A_t\right)_{t=0}^k$ and $J_C$ is as in Definition~\ref{rem:007}. \textcolor{black}{Any empty $k$-precorner in $\overline{X}$ results from a $k$-chain in $\widehat{X}$ that is closed by some element $g\in F$.}  By Lemma~\ref{lem:016}, any closing element in $F$ is conjugate to some element in $\mathcal{O}$. By Lemma~\ref{lem:015}, $\widehat{X}$ admits only trivial $k$-precorners \textcolor{black}{where each trivial $k$-precorner is over a $k$-corner that lifts into a single vertex-space of $\widehat{X}$. By assumption, the vertex-spaces of $\widehat{X}$ are nonpositively curved and thus contain only nonempty $k$-corners. Thus $\widehat{X}$ contains no closed $k$-chains and so}, $1_F\notin \mathcal{O}$. By Theorem~\ref{thm:RZ}, there exists a finite index normal subgroup $G\triangleleft F$ that is disjoint from $\mathcal{O}$. Let $\overline{X}=G\backslash\widehat{X}\rightarrow G\backslash\Gamma_{\widehat{X}}=\Gamma_{\overline{X}}$ and $\widehat{X}\rightarrow \overline{X}$ be the corresponding compact quotient and the regular covering map, respectively. \textcolor{black}{By Remark~\ref{rem:closing}, $\overline{X}$ contains no  empty $k$-precorners}, and thus  the horizontal quotient $\overline{X}^{E}$ \textcolor{black}{has no empty $k$-corners. By Remark~\ref{rem:004}, $\overline{X}^{E}$} is nonpositively curved.

\textcolor{black}{Finally, we note that any finite index normal subgroup of $G$ contains no closing elements and so, the corresponding finite cover splits as a graph of spaces with nonpositively curved horizontal quotient.}
\end{proof}

\section{The Construction}\label{sec:construction}

\begin{defn}\label{defn:localiso} Let $Y$ be a compact nonpositively curved cube complex, and let $Y'\subset Y$ be a subcomplex. The map $\varphi: Y'\subset Y\rightarrow Y$ is a \textit{partial local isometry} if $\varphi$ is a local isometry and both $Y'$ and $\varphi\left(Y'\right)$ are locally convex subcomplexes of $Y$.
\end{defn}
\begin{defn}
Let $Y$ be a nonpositively curved cube complex and let\\ $\mathcal{O}=\left\{\varphi_j:Y_{j}\subset Y\rightarrow Y\right\}_{j=1}^n$ be a collection of injective partial local isometries of $Y$ \textcolor{black}{where each $Y_j$ is connected}. The \textit{realization} of the pair $\left(Y,\mathcal{O}\right)$ is the cube complex $X$ obtained as the following quotient space:
\begin{center}
$X\ =\ \bigslant{Y \displaystyle\bigsqcup_{j=1}^{n} \left(Y_{j}\times I\right)}{\left\{\left(y,0\right)\sim y,\ \left(y,1\right)\sim\varphi_j\left(y\right),\ \forall\ y\in Y_{j}\right\}_{j=1}^n}$
\end{center}

The space $X$ decomposes as a graph of spaces via the map $X\rightarrow B$ with $Y\mapsto v$ and $Y_{j}\times I\mapsto \gamma_j$ where $B$ is the bouquet of $n$ circles $\left\{\gamma_j\right\}_{j=1}^n$ incident to a vertex $v$. 
\end{defn}
\begin{lem}\label{lem:002}
Let $\overline{X}\rightarrow \Gamma_{\overline{X}}$ be a compact graph of \textcolor{black}{cube complexes} with a strict horizontal quotient $\overline{X}\rightarrow \overline{X}^{E}$ \textcolor{black}{and isomorphic vertex-spaces.} \textcolor{black}{Let $\Phi\in \Aut\left(\Gamma_{\overline{X}}\right)$ and let $\overline{\Phi}\in \Aut\left(\overline{X}\right) $ be a combinatorial automorphism that maps vertex-spaces to vertex-spaces isometrically. Suppose that the left square of the diagram below commutes.}
Then \textcolor{black}{there exists an automorphism $\overline{\Phi}^{E}\in \Aut\left(\overline{X}^{E}\right)$ such that the right square of the diagram below commutes:}
\begin{center}
\begin{tikzcd}
\Gamma_{\overline{X}} \arrow[d, "\Phi"'] & \overline{X} \arrow[r, "q"] \arrow[d, "\overline{\Phi}"'] \arrow[l] & \overline{X}^{E} \arrow[d, "\overline{\Phi}^{E}"', dashed] \\
\Gamma_{\overline{X}}                    & \overline{X} \arrow[r, "q"] \arrow[l]                    & \overline{X}^{E}                           
\end{tikzcd}
\end{center}
\end{lem}
\begin{proof}
 \textcolor{black}{Define $\overline{\Phi}^{E}:\overline{X}^{E}\rightarrow \overline{X}^{E}$ by $\overline{\Phi}^{E}\left(y\right)=q\left(\overline{\Phi}\left(q^{-1}\left(y\right)\right)\right)$. Then $\overline{\Phi}^{E}$ is well-defined. Indeed, $q^{-1}\left(y\right)\subset \overline{X}$ is either a point or a horizontal graph. Since $\overline{\Phi}$ is a combinatorial automorphism, it maps points to points and (by the commutativity of the left square) horizontal graphs to horizontal graphs. In both cases, $q\left(\overline{\Phi}\left(q^{-1}\left(y\right)\right)\right)$ is a single point. Moreover, for each point $x\in \overline{X}$, we have $\overline{\Phi}^{E}\left(q\left(x\right)\right)=q\left(\overline{\Phi}\left(q^{-1}\left(q\left(x\right)\right)\right)\right)$. Since $q\left(x\right)$ is a point, $q^{-1}\left(q\left(x\right)\right)$ is either the point $x$ or a horizontal graph containing $x$. In both cases, $\overline{\Phi}^{E}\left(q\left(x\right)\right)=q\left(\overline{\Phi}\left(q^{-1}\left(q\left(x\right)\right)\right)\right)=q\left(\overline{\Phi}\left(\left(x\right)\right)\right)$ and thus the right square commutes. By the commutativity of the left square, $\overline{\Phi}$ permutes the vertex-spaces of $\overline{X}$ which makes $\overline{\Phi}^{E}$ an automorphism of $\overline{X}^{E}$ that permutes copies of the vertex-spaces.} 
\end{proof}

 \begin{thm}\label{thm:001}
Let $Y$ be a compact nonpositively curved cube complex and let $\mathcal{O}$ be the set of injective partial local isometries of $Y$. Then $Y$ embeds in a compact nonpositively curved cube complex $R$ where each $\varphi\in \mathcal{O}$ extends to an automorphism $\Phi\in \Aut\left(R\right)$.
\end{thm}
\begin{proof}
We construct a compact graph of spaces $\overline{X}$ whose horizontal quotient $\overline{X}^{E}=R$ has the desired properties.

Let $\mathcal{O}=\left\{\varphi_j:Y_{j}\subset Y\rightarrow Y\right\}_{j=1}^n$ be the collection of injective partial local isometries of $Y$ and let $X\rightarrow B$ be the realization of the pair $\left(Y,\mathcal{O}\right)$. Let $\gamma_j\rightarrow B$ be the closed path giving the loop in $B$ that corresponds to $\varphi_j$. Let $F=\pi_1B$ and let $\widehat{X}\rightarrow X$ be the covering map induced by the universal covering $\widetilde{B}\rightarrow B$ such that the following diagram commutes:
\begin{center}
\begin{tikzcd}
\widehat {X} \arrow[r] \arrow[d] & \widetilde{B} \arrow[d] \\
X \arrow[r]                      & B                                           
\end{tikzcd}
\end{center}
Then $\widehat{X}\rightarrow \Gamma_{\widehat{X}}=\widetilde{B}$ is a nonpositively curved tree of cube complexes. By  Lemma~\ref{lem:012} and Lemma~\ref{lem:017}, there exists a finite regular cover $\overline{X}\rightarrow X$ that splits as a graph of spaces according to the following commutative diagram and such that the horizontal quotient $\overline{X}\rightarrow\overline{X}^{E}$ is \textcolor{black}{strict and $\overline{X}^{E}$ is} nonpositively curved. Note that each vertex-space of $\overline{X}$ is \textcolor{black}{a \emph{standard copy} of $Y$ according to some fixed isomorphism. We will use the term ``copy'' below to mean such a standard copy}.
\begin{center}
\begin{tikzcd}
\widehat {X} \arrow[r] \arrow[d] & \Gamma_{\widehat{X}} \arrow[d] \\
\overline{X} \arrow[r] \arrow[d] & \Gamma_{\overline{X}} \arrow[d]              \\
X \arrow[r]                      & B                                           
\end{tikzcd}
\end{center}
Fix a \textcolor{black}{vertex $v\in \Gamma_{\overline{X}}$ and let $\overline{X}_v$ be the corresponding vertex-space of $\overline{X}$}.  By subgroup separability of free groups, we can assume that $\Gamma_{\overline{X}}$ has no loops. Thus $\overline{X}_v$ is adjacent to $2n$ vertex-spaces $\left\{\overline{X}_{v_i}\right\}_{i=1}^{2n}$. \textcolor{black}{For each $\varphi_j\in \mathcal{O}$, there are two vertex-spaces $\overline{X}_{v_j}$ and $\textcolor{black}{\overline{X}_{v_{2j}}}$ that are joined to $\overline{X}_v$ by copies of $Y_j\times [-1,1]$ where $Y_j\times [-1,1]$ joins a 
copy of $\varphi_j\left(Y_j\right)$ in $\overline{X}_v$ to a  copy of $Y_j$ in $\overline{X}_{v_j}$, and it joins a  copy of $Y_j$ in $\overline{X}_{v}$ to a  copy of $\varphi_j\left(Y_j\right)$ in $\overline{X}_{v_{2j}}$. Each $Y_j\times [-1,1]$ corresponds to a unique map $\varphi_j\in \mathcal{O}$ and thus to a unique closed path $\gamma_j\rightarrow B$. The lift of $\gamma_j$ at $v$ specifies a unique automorphism $\Phi_j\in \Aut\left(\Gamma_{\overline{X}}\right)$ that maps $v$ to $v_j$. Then there is an automorphism $\overline{\Phi}_j\in \Aut\left(\overline{X}\right)$ that maps $\overline{X}_v$ to $\overline{X}_{v_j}$ such that the following diagram commutes:}
\begin{center}
\begin{tikzcd}
\overline{X} \arrow[r, "\overline{\Phi}_j"] \arrow[d] & \overline{X} \arrow[d] \\
\Gamma_{\overline{X}} \arrow[r, "\Phi_j"]             & \Gamma_{\overline{X}}   
\end{tikzcd}
\end{center}
\textcolor{black}{Note that $\overline{\Phi}_j$ maps copies of $Y_j$ and $\varphi_j\left(Y_j\right)$ in $\overline{X}_v$ to copies of $Y_j$ and $\varphi_j\left(Y_j\right)$ in $\overline{X}_{v_j}$, respectively. However, in $\overline{X}^{E}$ the copy of $Y_j$ in $ q\left(\overline{X}_{v_j}\right)=q\left(\overline{\Phi}_j\left(\overline{X}_{v}\right)\right)$ is identified with a copy of $\varphi_j\left(Y_j\right)$ in $q\left(\overline{X}_v\right)$. By Lemma~\ref{lem:002}, any automorphism $\overline{\Phi}\in \Aut\left(\overline{X}\right)$ induced by an automorphism of the underlying graph $\Phi\in \Aut\left(\Gamma_{\overline{X}}\right)$ descends to an automorphism $\overline{\Phi}^{E}\in \Aut\left(\overline{X}^{E}\right)$. So $\overline{\Phi}_j^{E}\left(q\left(\overline{X}_{v}\right)\right)=q\left(\overline{\Phi}_j\left(\overline{X}_v\right)\right)=q\left(\overline{X}_{v_j}\right)$. Note that $\overline{\Phi}_j^{E}$ maps a copy of $Y_j$ in $q\left(\overline{X}_{v}\right)$ to a copy of  $\varphi_j\left(Y_j\right)$ in $q\left(\overline{X}_{v}\right)$. Identify $Y$ with $q\left(\overline{X}_v\right)$ and assume $\varphi_j:Y_j\subset q\left(\overline{X}_v\right)\rightarrow q\left(\overline{X}_v\right)$. Then $Y$ embeds in $\overline{X}^{E}$ and the restriction $\overline{\Phi}^{E}_j|_{Y_j}=\varphi_j$.} 
\end{proof}
\begin{rem}
Note that $\dim\left(\overline{X}^{E}\right)=\dim\left(Y\right)$.
\end{rem}
\begin{rem}\label{metric}
Following the \textit{Simple Local Gluing} Lemma in~\cite{BridsonHaefliger}, Theorem~\ref{thm:001} can be generalized to nonpositively curved metric spaces provided that some finiteness conditions are satisfied and the edge-spaces are locally convex, closed, and complete subspaces.
\end{rem}

\begin{defn}\label{defn:controlled}
Let $Y$ be a compact nonpositively curved cube complex. A collection of injective partial local isometries $\mathcal{O}=\left\{\varphi_j: Y_j\subset Y\rightarrow Y\right\}_{j=1}^n$ is \textit{controlled} if the  corresponding realization $X\rightarrow B$ is a controlled graph of spaces.
\end{defn}

\begin{thm}[Haglund-Wise \cite{HaglundWiseCoxeter}]\label{hw2}
Let $X$ decompose as a finite graph of spaces, where each vertex-space $X_v$ and edge-space $X_e$ is special with finitely many hyperplanes. Then $X$ has a finite special cover provided
the attaching maps of edge-spaces satisfy the following:
\begin{enumerate}
\item the attaching maps $X_e\rightarrow X_{\iota\left(e\right)}$  and $X_e\rightarrow X_{\tau\left(e\right)}$ are injective local-isometries
\item distinct hyperplanes of $X_e$ map to distinct hyperplanes of $X_{\iota\left(e\right)}$ and $X_{\tau\left(e\right)}$
\item non-crossing hyperplanes map to non-crossing hyperplanes
\item no hyperplane of $X_e$ extends in $X_{\iota\left(e\right)}$ or $X_{\tau\left(e\right)}$ to a hyperplane dual to an edge that intersects $X_e$ in a single vertex.
\end{enumerate}
\end{thm}

\begin{rem}\label{remark}
\textcolor{black}{The finite special cover in Theorem~\ref{hw2} corresponds to a finite index subgroup $N$ of the fundamental group of the underlying graph of $X$ where any subgroup of $N$ induces a cover that is special. See \cite{HaglundWiseCoxeter} for details.  This makes Theorem~\ref{hw2} compatible with the prevailing methods of finding finite covers in this text, namely finding covers of graphs of spaces induced by finite covers of their underlying graphs}.
\end{rem}

\begin{thm}\label{thm:002}
Let $Y$ be a compact special cube complex and let $\mathcal{O}$ be a controlled collection of injective partial local isometries of $Y$. Then  there exists a compact special cube complex $R$ containing $Y$ as a locally convex subcomplex such that each $\varphi\in \mathcal{O}$ extends to some automorphism $\Phi\in \Aut\left(R\right)$.
\end{thm}
\begin{proof}
 \textcolor{black}{Let $X$ be the realization of the pair $\left\{Y,\mathcal{O}\right\}$. Then 
$X$ is a cube complex that splits as a graph of spaces $X\rightarrow \Gamma_X$ where  
$\Gamma_X$ is a compact bouquet of circles and $\pi_1\Gamma_X=F$ is a free group. Since $Y$ is compact and $\mathcal{O}$ is controlled, $X$ is a compact controlled graph of spaces. The claim then follows from Remark~\ref{remark}  Theorem~\ref{hw2}, Theorem~\ref{thm:001}  Lemma~\ref{no remote osculation},
Remark~\ref{rem:special},
Lemma~\ref{lem:013}, and Lemma~\ref{lem:003}.}
\end{proof}

\bibliographystyle{alpha}
\bibliography{wise.bib}

\end{document}